\documentclass{article}
\usepackage{amssymb,amsfonts,amsthm,graphicx,wrapfig,fancyhdr,fullpage,pst-barcode}
\newtheorem{lma}{Lemma}
\newtheorem*{pps}{Proposition}
\newtheorem*{thm}{Theorem}
\newtheorem{crl}{Corollary}
\theoremstyle{remark}
\newtheorem{rmk}{Remark}
\title{Decomposition Algorithm for Median Graph of Triangulation of a Bordered 2D Surface}
\author{Weiwen Gu}
\date{}
\begin{document}
\pagenumbering{roman}
\tableofcontents
\newpage
\pagenumbering{arabic}
\setcounter{page}{1}
\maketitle
\begin{abstract}
This paper develops an algorithm that identifies and decomposes a
median graph of a triangulation of a 2-dimensional (2D) oriented
bordered surface and in addition restores all corresponding
triangulation whenever they exist. The algorithm is based on the
consecutive simplification of the given graph by reducing degrees of
its nodes. From the paper \cite{FST1}, it is known that such graphs
can not have nodes of degrees above 8. Neighborhood of nodes of degrees 8,7,6,5, and
4 are consecutively simplified. Then, a criterion is provided to
identify median graphs with nodes of degrees at most 3. As a
byproduct, we produce an algorithm that is more effective than
previous known to determine quivers of finite mutation type of size
greater than 10.
\end{abstract}
\section[Introduction]{Introduction}
Triangulations of 2D surface are instrumental for in many math theories. To mention a few, they are
used in \cite{NI} to construct coordinates on Teichm$\ddot{u}$ller Space. Also, in
\cite{FST1}, the authors uses triangulations of a 2D surface to construct a cluster algebra.
Moreover, they described a principal way to determine if a cluster algebra is originated from
a surface triangulation. The method uses the idea of block decomposition of the median graph of triangulation.\\\\
Block decomposition also plays an important role in
determining the mutation class of a \textit{quiver}. A
\textit{quiver} is defined as a finite oriented multi-graph without
loops and 2-cycles. Based on the cluster algebra constructed in
\cite{FST1}, a \textit{seed} ($f,B$) is defined, where $f$ is a
collection of $n$ algebraically independent rational functions of
$n$ variables and $B$ is a skew-symmetrizable matrix. Cluster
algebra formalism introduces a certain operation on seeds. This
operation is called \textit{mutation}, see \cite{FST2} Definition
2.1. Two seeds are said to be \textit{mutation equivalent} if one is
obtained from the other by a sequence of seed mutation. A
\textit{mutation class} is the collection of mutation-equivalent
seeds. In \cite{FST1}, the authors prove that the mutation class of
any \textit{block-decomposable} quiver is finite.\\\\
In this paper, we provide a combinatorial algorithm
that determines if a given graph is decomposable.
Moreover, the algorithm also determines
all corresponding triangulations for decomposable graphs.\\\\
A block is a directed graph that is isomorphic to one of the graphs
shown in Figure.\ref{Decomposition}. They are categorized as one of
the following: type I (\textit{spike}), II (\textit{triangle}), IIIa (\textit{infork}),
IIIb (\textit{outfork}), IV (\textit{diamond}), and V (\textit{square}).
The nodes marked by unfilled circles are
called \textit{outlets} or \textit{white nodes}. The nodes marked by filled circles
are called \textit{dead ends} or \textit{black nodes}. A directed graph G is called
\textit{block decomposable} or simply \textit{decomposable} if it
can be obtained from disjoint blocks as a result of the following gluing rules: (See \cite{FST1} for definition.)
\begin{table}
\centering
\begin{tabular}{cccccc}
&&&&&\\
\includegraphics{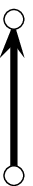}& \includegraphics{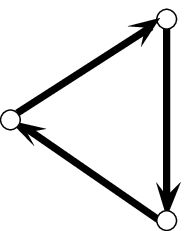}& \includegraphics{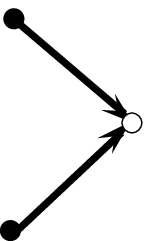} &\includegraphics{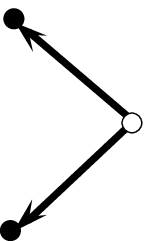} & \includegraphics{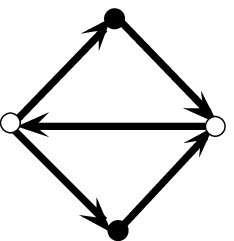}&\includegraphics{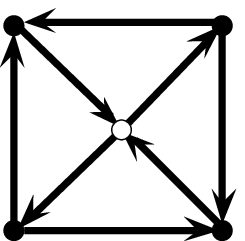}\\&&&&&\\
$\textbf{I}$:Spike&$\textbf{II}$:Triangle&$\textbf{III}$a:Infork&$\textbf{III}$b:Outfork&$\textbf{IV}$:Diamond&$\textbf{V}$:Square
\end{tabular}
\caption{Blocks}\label{Decomposition}
\end{table}
\begin{enumerate}
\item Two white nodes of two different blocks can be identified. As a result, the graph becomes a union of two parts.
The common node becomes black. A white node can not be identified with
another node of the same block. See Figure
\ref{ex1}--\ref{ex1ill}.
\item A black node can not be identified with any other node.
\item If an edge $a=x\rightarrow y$ with two white nodes ($x,y$) is glued to another
edge $b=p\rightarrow q$ with two white nodes ($p,q$) in the
following way: $x$ is glued to $p$ and $y$ is glued to $q$, then a
multi-edge is formed, and the nodes $x=p$, $y=q$ become black.
(Figure \ref{multi})
\item If an edge $a=x\rightarrow y$ with two white nodes $x,y$ is glued to another edge $b=q\rightarrow p$ in the following way: $x$ is glued to $p$ and $y$ is glued to $q$,
then both edges are removed after gluing, the nodes $x=p$, $y=q$
become black. We say that edges annihilate each other. (Figure
\ref{anni})
\end{enumerate}
\begin{figure}[tbh]
\centering
 \begin{minipage}[c]{0.4\linewidth}
 \centering
 \includegraphics[width=0.4\linewidth]{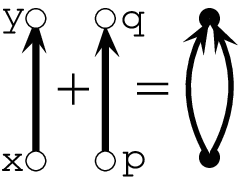}
 \caption{}\label{multi}
 \end{minipage}
 \begin{minipage}[c]{0.4\linewidth}
 \centering
 \includegraphics[width=0.4\linewidth]{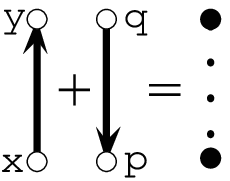}
 \caption{}\label{anni}
 \end{minipage}
\end{figure}
\noindent For example, Figure \ref{ex1} can be constructed from an
infork ($\textbf{III}$a) and a spike~($\textbf{I}$). (Figure
\ref{ex1decom}).
\begin{rmk}By design, a block-decomposable graph has
no loop and all edge multiplicities are 1 or 2.
\end{rmk}
\begin{figure}[thb]
\centering
\begin{minipage}[t]{0.3\linewidth}
\centering
\includegraphics[width=0.6\linewidth]{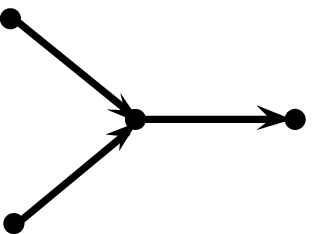}
\caption{Example 1}\label{ex1}
\end{minipage}
\begin{minipage}[t]{0.3\linewidth}
\centering
\includegraphics[width=0.6\linewidth]{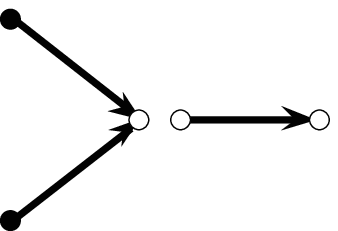}
\caption{Decomposition of Example 1}\label{ex1decom}
\end{minipage}
\begin{minipage}[t]{0.3\linewidth}
\centering
\includegraphics[width=0.6\linewidth]{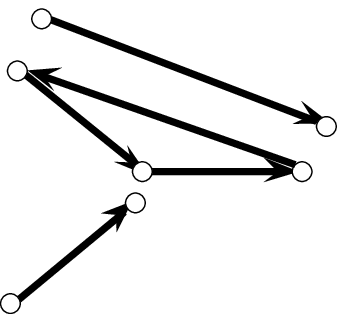}
\caption{Another way to decompose Figure.\ref{ex1}}\label{ex1ill}
\end{minipage}
\end{figure}
\begin{rmk}Note that in the algorithm, the color of a vertex is not specified in the
original graph, and is determined only for vertices of blocks of a
specified block decomposition. (There might be several ways to decompose
a graph. Hence, a vertex may have different colors in different
decompositions, see Figure \ref{ex1},\ref{ex1decom} and
\ref{ex1ill}.)
\end{rmk}
\noindent We will assume in the following discussion that $G$ is a finite
oriented multi-graph without loops and 2-cycles.
\begin{pps} A graph $G$ without isolated nodes is decomposable if and only if every disjoint connected component is decomposable.
\end{pps}
\begin{proof}
It's suffice to show that annihilating
an edge in a connected graph generates a connected graph. Since we can only annihilated edges in a spike, triangle or diamond block,
and before annihilating an edge, both of its endpoint must be white. Denote these two endpoints by $x,y$.
\begin{itemize}
\item Suppose $x,y$ are endpoints of a spike. Notice that the original graph must be a single spike. If we annihilate the edge by gluing a triangle, $x,y$ will be connected via the third node of the triangle. If we annihilate the edge by gluing a diamond, $x,y$ will be connected via the remaining nodes of the diamond. If we annihilate it by gluing a spike, the new graph will consist only two nodes and no edge. Contradiction.
\item Suppose $x,y$ are endpoints of a triangle. If we annihilate the edge $\overline{xy}$ by gluing a spike or diamond, the remaining two edges of the triangle can not be annihilated, and $x,y$ will still be connected via the third node of the triangular block. If we glue another triangle, there are two cases: we can annihilate only one edge, namely $\overline{xy}$. Then $x,y$ will still be connected via the third node. We can also annihilate the whole triangle when all three nodes are white. In this case, the original graph is a single triangular block and the new graph is simply three nodes. This is again a contradiction.
\item Suppose $x,y$ are endpoints of a diamond. Since none of the boundary edges can be annihilated, after gluing a spike or triangle or another diamond to edge $\overline{xy}$, $x,y$ will still be connected.
\end{itemize}
\end{proof}
According to the above proposition, if $G$ is decomposable, we may break connectivity in a graph only in two trivial cases. In both cases, the resulting graph contains isolated nodes. On the other hand, in a decomposable graph, isolated nodes can only be obtained in the above manner. In particular, the decomposition of subset of $G$ consisting isolated nodes is independent with the rest of $G$. Therefor, we can assume from now on that the graph is connected.\\\\
Notice that the highest degree of any node in a block is 4. If
$G$ is decomposable, the highest degree of any node in $G$ does not
exceed 8.\\
In the algorithm, for every graph, the
neighborhoods of nodes of highest degree (at most 8) are simplified. The
neighborhoods of all such nodes are analyzed and replaced by simpler
ones so that the degrees of these nodes decrease. After the
neighborhoods of nodes of degree 8 are exhausted, proceed to the
nodes of degree 7, then, 6,5 and 4. This paper proves that all
replacements are \textit{reversible}: the original graph is decomposable if and only if the new graph is. At last, we give a theorem that identifies decomposable graphs containing nodes of degree at most 3.\\
\section[Simplification]{Simplification on Nodes of Degree 8,7,6 and 5}
In this section we show how to replace the neighborhood of certain node by
an equivalent one which decreases the degree of this node. As a result, the nodes of degree larger than four
are eliminated, consecutively.
\subsection{Nodes of Degree 8}
A node $o$ of degree 8 in decomposable graph $G$ can only be
obtained by gluing a square with another square. (See
Figure.\ref{deg8}) The result is a disjoint connected component.
Otherwise, $G$ is undecomposable.
\subsection{Nodes of Degree 7}
If $o$ is a node of degree 7 in a decomposable graph $G$, it must
have resulted from gluing together a diamond and a square, see
Figure.\ref{deg7}. The neighborhood is replaced by the one in Figure \ref{deg7simp}. The Lemma \ref{trianglespike} shows that this
replacement is reversible.\\
\begin{lma}
Suppose the neighborhood of node $o$ is as in Figure \ref{deg7simp}. nodes $x,p$ and nodes $y,p$ are disconnected. If $G$ is decomposable, the neighborhood can only be decomposed into a triangle and a
spike.\label{trianglespike}
\end{lma}
\begin{proof}
It is necessary to show that $b,c,d$ form a triangular block in the decomposition and $a$ comes from a spike block.\\
Assume there is a decomposition, I claim that the block containing
$b$
must be a triangle.\\
Suppose that the claim is false.
\begin{enumerate}
\item Suppose that $b$ comes from a fork. Since both edges in a fork
contain black endpoints, they can not be annihilated. Thus, the fork
containing $b$ must also contain $a$ or $c$. However, the directions
of $a$ and $c$ are not compatible with the directions of edges in any
fork block. Therefore, $b$ can not be a part of a fork.
\item Suppose $b$ comes from a square block. Since at least one endpoint of any
edge in a square block is black, none of the edges can be annihilated. Thus, the
degree of any corner node is 3, and the central node has degree at least 4.
Since the degree of node $o$ is 3, it can only be one of the corner node in the square.
Moreover, since nodes $x$ and $p$ are not connected, they must both corner nodes on the same diagonal.
Therefore, node $y$ must be the central node. Hence nodes $y$ and $p$ must be connected, which is a contradiction. Hence $b$ can not come from a square.
\item Suppose that $b$ comes from a diamond. If the diamond does not
contain $c$ or $d$, then it is necessary to
glue $d$ and $c$ together. Since the only white nodes are the
endpoints of mid-edge, $b$ must be the mid-edge of the diamond. Suppose the diamond does not contain $c$. The edges $a,d$ must be both contained in the diamond since the degree of node $o$ is 3. Hence nodes $x,p$ must be connected, which is a contradiction. Suppose the diamond does not contain $d$, then after gluing $d$ to node $o$, the degree of $o$ must be at least 4. This again leads
to a contradiction. So the diamond must
contain $d$ and $c$. The directions of $b,c$
suggest that both $b,c$ are in the upper or lower triangle of the
diamond. This forces $d$ to be contained in the diamond. Notice that
node $x$ is not connected with $p$, the other half of the diamond is
annihilated. This is again a contradiction. So the diamond can not
contain $c$. To conclude, $b$ is not contained in a diamond block.
\item Suppose $b$ comes from a spike, $a,d$ must come from the same block. Thus,
they form a fork. Then, $c$ can not be attached. This proves the
claim.
\end{enumerate}
Now, the only option is that $b$ comes from a triangular block
$\triangle_1$. If $a$ also comes from $\triangle_1$, the third edge
in $\triangle_1$ should be annihilated by another edge, denoted by $e$.
Moreover, both $e$ and $c$ must be obtained from the same block. Taking
into account direction of edges, this block must be a triangle
$\triangle_2$. Note the third edge $\overline{py}$ of $\triangle_2$
is annihilated by an edge $f$ incident to node $y$, so $f$ and $d$
must come from the same block. Again, considering the directions of
edges, this block must also be a triangle $\triangle_3$. Therefore,
the third edge of $\triangle_3$ must be $a$. This contradicts
to the assumption that $a$ is an edge of block
$\triangle_1$. Therefore $a$ is not contained in $\triangle_1$ and
this triangle is formed by $b,c,d$. This forces $a$ to be a spike block.
\end{proof}
%====================Need Correction===============
\begin{rmk}\label{equiv}
After the original neighborhood is replaced by the on in Figure
\ref{deg7simp}, assume the new graph is not decomposable. This means
that if the lower triangle and spike described in Lemma
\ref{trianglespike} are removed, the rest is not decomposable.
Therefore, in the original graph, after the original neighborhood of
$o$ is removed, the graph is undecomposable. However, in this case,
the neighborhood of $o$ can only be obtained from gluing a square
and a diamond. Hence the original graph is non-decomposable. This
proves that the replacement is reversible. Moreover, all
decomposition of the original graph are in 1-1 correspondence with
decomposition of the new one.
\end{rmk}
%=====================End of Correction
\vspace{1em}
\begin{figure}[t]
\centering
\begin{minipage}[c]{0.3\linewidth}
\centering
\includegraphics[width=0.5\linewidth]{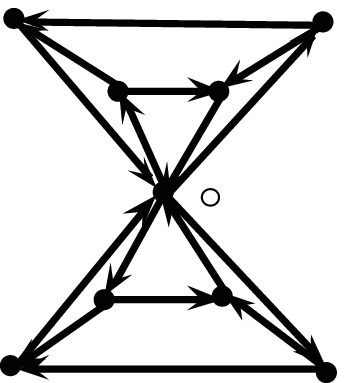}
\caption{Node of degree 8}\label{deg8}
\end{minipage}
\begin{minipage}[c]{0.3\linewidth} \centering
\includegraphics[width=0.5\linewidth]{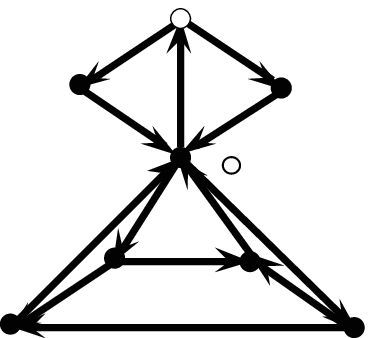}
\hspace{0.5em}
\includegraphics[width=0.5\linewidth]{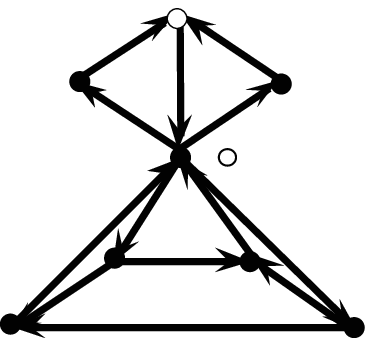}
\caption{Node of Degree 7}\label{deg7}
\end{minipage}
\begin{minipage}[c]{0.3\linewidth}
\centering \vspace{-1em}
\includegraphics[width=0.7\linewidth]{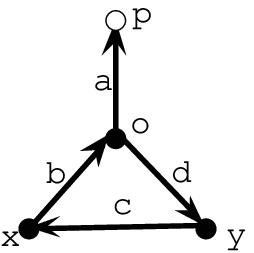}
\caption{}\label{deg7simp}
\end{minipage}
\end{figure}
\subsection{Nodes of Degree 6}
If $o$ is a node in $G$ of degree 6, there are three cases:\\
\begin{enumerate}
\item The neighborhood of $o$ comes from a triangle and a square block. (Figure \ref{deg6_1})
Then replace it by the one in Figure \ref{B11s}. Lemma
\ref{sandwichlemma} shows that this replacement is
reversible.\vspace{-0.5em}
\item The neighborhood of $o$ comes from a fork block (infork or outfork) and a square block.
(Figure \ref{deg6_2}). Then the neighborhood is a disjoint connected
component, otherwise the graph is undecomposable.\vspace{-0.5em}
\item The neighborhood of $o$ comes from two diamonds, as illustrated in
Figure \ref{doublediamonds}. Note that in Figure
\ref{doublediamonds}A, if node $p$ and node $q$ are glued together,
the result, if decomposable, is a disjoint connected component, see
Figure \ref{ddglue}A. Otherwise, $G$ is undecomposable. If $p$, $q$
are not glued together, the neighborhood is then replaced by the
graph in Figure \ref{doublediamondsimp1}. Lemma \ref{trianglespike}
shows that this replacement is reversible. On the other hand, gluing
nodes $p,q$ together in Figure \ref{doublediamonds}B, the mid-edges
are annihilated, as seen in Figure \ref{ddglue}B. Hence, the degree
of $o$ must be 4. This contradicts the fact that node $o$ has degree 6.
In this case, the neighborhood is replaced by Figure
\ref{doublediamondsimp1}.
\end{enumerate}
\begin{crl}
Figure \ref{B11s} can only be decomposed as a triangle plus a spike
plus a triangle.\label{sandwichlemma}
\end{crl}
\begin{proof}
By Lemma \ref{trianglespike}, the lower part of Figure \ref{B11s} can only be obtained from gluing a spike and a triangle.
For the upper part, the two edges incident to $o$ must come from the same block. Judging by their directions, the block can only
be a triangle. Note that the
third edge of this triangle may be annihilated, as indicated by
dashed line in Figure \ref{B11s}.
\end{proof}
\begin{rmk}\label{rmk4}
By the similar argument as in \textit{Remark} \ref{equiv}, this
replacement is reversible.
\end{rmk}
\vspace{1em}
\begin{figure}[tbc]
 \centering
 \begin{minipage}[b]{0.25\linewidth}
 \centering
 \includegraphics[width=0.7\linewidth]{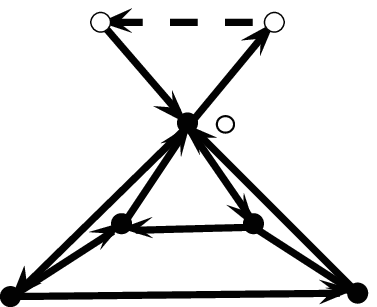}
 \caption{}\label{deg6_1}
 \end{minipage}
 \begin{minipage}[b]{0.35\linewidth}
 \centering \includegraphics[width=0.5\linewidth]{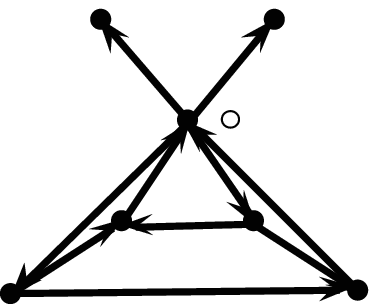}
 \includegraphics[width=0.5\linewidth]{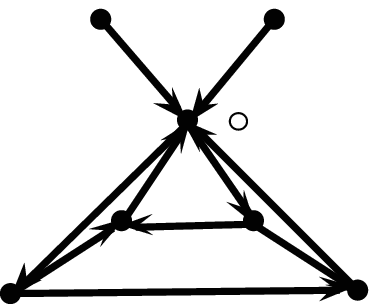}
 \caption{}\label{deg6_2}
 \end{minipage}
 \begin{minipage}[b]{0.3\linewidth}
 \centering
 \includegraphics[angle=90,width=0.5\linewidth]{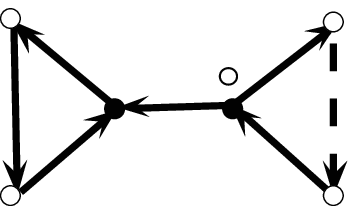}
 \caption{}\label{B11s}
 \end{minipage}
\end{figure}
\begin{figure}[tbh]
\centering
 \begin{minipage}[b]{0.45\linewidth}
 \centering
 \includegraphics[width=0.4\linewidth]{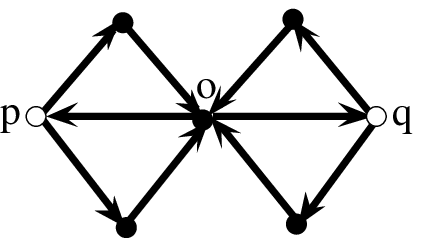}
 \includegraphics[width=0.4\linewidth]{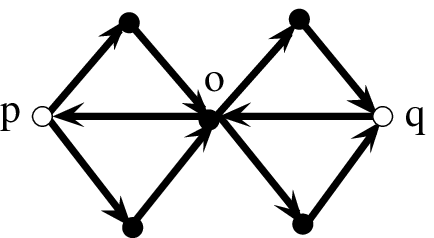}\\
 \mbox{A\hspace{1.8cm}B}
 \caption{}\label{doublediamonds}
 \end{minipage}
 \begin{minipage}[b]{0.45\linewidth}
 \centering
 \includegraphics[width=0.4\linewidth]{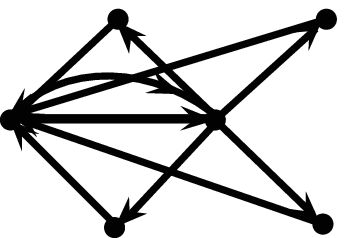}
 \includegraphics[width=0.4\linewidth]{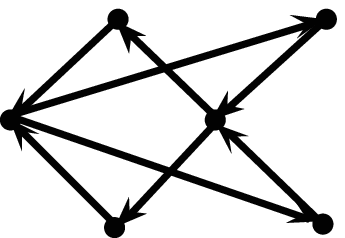}
 \mbox{A\hspace{1.8cm}B}
 \caption{}\label{ddglue}
 \end{minipage}
\end{figure}
\begin{figure}
\centering
 \begin{minipage}[b]{0.5\linewidth}
 \centering
 \includegraphics[width=0.8\linewidth]{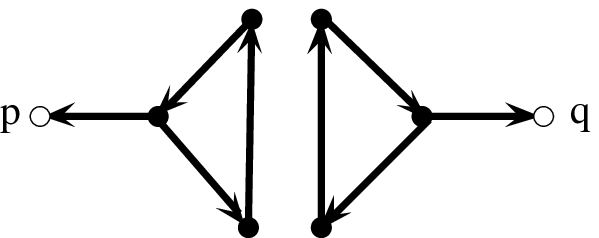}\\
 \mbox{ }
 \caption{}\label{doublediamondsimp1}
 \end{minipage}\\
\end{figure}
\subsection{Nodes of Degree 5}
If node $o$ has degree 5, there are three cases:
\begin{enumerate}
\item The neighborhood of $o$ comes from a spike and a square, see Figure
\ref{spikesquare}. In this case, we replace it with the neighborhood
in Figure \ref{deg7simp}. According to Lemma \ref{trianglespike},
the replacement is reversible.
\item The neighborhood of $o$ comes from a fork and a diamond. See Figure \ref{deg5fd}.
Note that the direction of the fork and the diamond can change, so
there are 4 subcases. In all of these cases replace the
neighborhoods by the one in Figure \ref{deg5fdsimp}. The replacement
is reversible due to Lemma \ref{trianglespike} and \textit{Remark}
\ref{equiv}.
\item The neighborhood of $o$ comes from a triangle and a diamond. See
Figure \ref{deg5td}. Similarly, note that the orientation of both the
triangle and the diamond can also be reversed, there are 4
possible neighborhoods in this case. Up to direction reversions, the neighborhood is replaced by the one
in Figure \ref{deg5tdsimp}. Lemma \ref{trianglespike} and Corollary
\ref{sandwichlemma} ensure that this replacement is reversible.
\end{enumerate}

\begin{figure}[tb]
 \centering
 \begin{minipage}[b]{0.24\linewidth}
 \centering
 \includegraphics[width=0.6\linewidth]{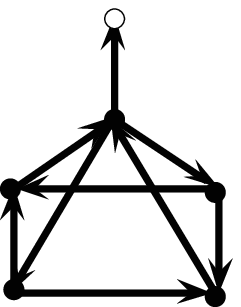}
 \caption{}\label{spikesquare}
 \end{minipage}
 \begin{minipage}[b]{0.24\linewidth}
 \centering
 \includegraphics[width=0.6\linewidth]{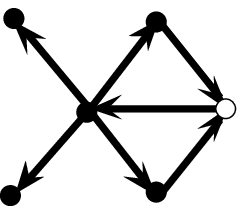}
 \caption{}\label{deg5fd}
 \end{minipage}
 \begin{minipage}[b]{0.24\linewidth}
 \centering
 \includegraphics[width=0.7\linewidth]{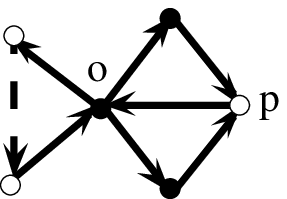}
 \caption{}\label{deg5td}
 \end{minipage}
 \begin{minipage}[b]{0.24\linewidth}
 \centering
 \includegraphics[width=0.6\linewidth]{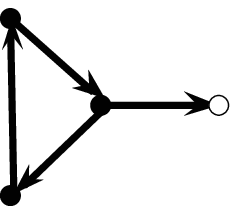}
 \caption{}\label{deg5fdsimp}
 \end{minipage}\\
 \vspace{1em}
 \begin{minipage}[b]{0.5\linewidth}
 \centering
 \includegraphics[width=0.7\linewidth]{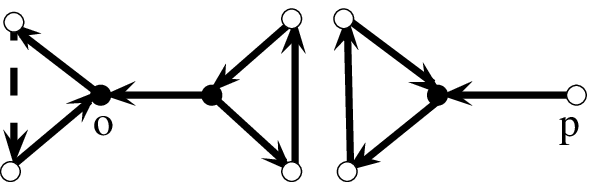}
 \caption{}\label{deg5tdsimp}
 \end{minipage}
\end{figure}

%============================From Decomposition===============
%\section{Fundamental}
%\begin{figure}[t]
%  \centering
%  \begin{minipage}[t]{0.16\linewidth}
%  \centering
%  \includegraphics{spike.eps}
%  \addtocontents{toc}{AA}
  %\caption{$\mathcal{B}_I$}
  %\label{Sp}
% \end{minipage}%
% \begin{minipage}[t]{0.16\linewidth}
%  \centering
%  \includegraphics{triangle.eps}
  %\caption{$\mathcal{B}_{II}$}
  %\label{Tr}
% \end{minipage}
% \begin{minipage}[t]{0.16\linewidth}
%  \centering
%  \includegraphics{infork.eps}
  %\caption{$\mathcal{B}_{IIIa}$}
  %\label{In}
% \end{minipage}
% \begin{minipage}[t]{0.16\linewidth}
%    \centering
%    \includegraphics{outfork.eps}
  %\caption{$\mathcal{B}_{IIIb}$}
  %\label{Out}
% \end{minipage}
% \begin{minipage}[t]{0.16\linewidth}
%  \includegraphics{diamond.eps}
  %\caption{$\mathcal{B}_{IV}$}
  %\label{Di}
% \end{minipage}
%\caption{Components} \label{Compo}

%\end{figure}
%\begin{figure}
%\includegraphics{spike.eps}
%\caption{Sp}
%\includegraphics{triangle.eps
%\caption{Tr}}
%\end{figure}
%We have the following 6 types of graphs.(Figure \ref{Compo}). In
%this article we focus on how to classify, simplify or decompose at
%the Nodes of
%degree 4.We classify the degree 4 class into 3 cases.\\

\section[Nodes of Degree 4]{Simplification on Nodes of Degree 4}
After simplification in the previous section, assume now that all
nodes in graph $G$ have degrees at most 4. In this section we shall
denote the node in consideration by $o$ and the nodes connected to it.
we call them \textit{boundary nodes}. Note that taking into account directions of
edges incident to $o$ we can distinguish the following three cases:\\
\begin{enumerate}
\item[\textbf{A:}] 4 outward edges or 4 inward edges.
\item[\textbf{B:}] 3 outward edges $+$ 1 inward edge or 3
inward edges $+$ 1 outward edge.
\item[\textbf{C:}] 2 outward edges $+$ 2 inward edges.
\end{enumerate}
We shall consider all the situations above cases by case.
\subsection{4 outward edges or 4 inward edges.}
Without loss of generality, assume that there are 4 edges directed
outwards. If the graph is decomposable, there is only one case. Namely,
the neighborhood is obtained by the gluing of two forks as below. \\
\begin{figure}[h]
\centering
\begin{minipage}[c]{.25\linewidth}
 \centering
 \includegraphics[width=0.5\linewidth]{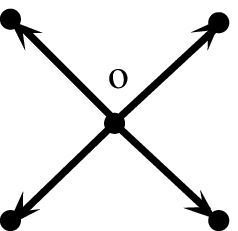}
\end{minipage}{=}
\begin{minipage}[c]{.25\linewidth}
\centering
\raisebox{-0.5em}{\includegraphics[width=0.3\linewidth]{outfork.eps}}
\end{minipage}{+}
\begin{minipage}[c]{.25\linewidth}
\centering
\includegraphics[angle=180,width=0.3\linewidth]{outfork.eps}
\end{minipage}
\caption{}\label{A1}
\end{figure}
\subsection{Three outward edges $+$ one inward edge or three
inward edges $+$ one outward edge}\label{node3o1in}
Without loss of generality, assume that $o$ is incident to three outward edges and one inward edge.\\\\
Assume that the only distinct from $o$ node $p$ that is incident to the incoming edge has degree at least two.\\
\begin{enumerate}
\item The inward edge can not be obtained from a fork. To show this, we use contradiction, suppose it's contained in a fork block, $o$
must be the white node in this block. Therefore, the other inward edge must be incident to $o$ and can not be annihilated.
This contradicts the fact that there is only one inward edge incident to $o$. Note that this
argument is still true even if $p$ has degree one.
\item Suppose the inward edge comes from a square. Since the degree of $o$ is 4,
it must be the center of the square and all four edges are contained
in the same square. This is impossible since none of the edges in a square can be annihilated
and the central node of a square is incident to at least two inward edges and two outward edges.
\item Assume the inward edge is a part of a triangular block. Suppose this
triangle does not contain any of the remaining three outward edges.
Then the other edge of the triangle which is incident to $o$ is
annihilated by another edge, denoted as $e$. In this case, $e$ and
the remaining three outward edges must come from the same block. It
can only be a square with central node $o$. On the other hand, $o$ is incident to three outward edges and one inward edge. This is a contradiction. Therefore, the triangle must contain one of the outward edges. This
forces the remaining two outward edges to be in the same block. This block must be a fork. See Figure \ref{B11}.
\item Assume that the inward edge comes from a diamond. Node $o$ must be a white node
in the diamond. Judging by the directions of the remaining edges,
two of them must be boundary edges of the same diamond, see Figure
\ref{B12}.
\item Suppose that the inward edge comes from a spike, then the remaining three
outward edges come from the same block. However there is no block that contains three outward edges incident to the same node.
Hence in this case, the graph is undecomposable.
\end{enumerate}
\begin{figure}[cbh]
 \centering
 \begin{minipage}[b]{0.24\linewidth}
 \centering
 \includegraphics[width=0.6\linewidth]{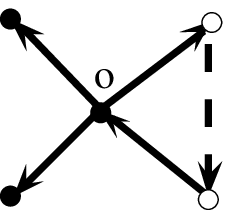}
 \caption{}\label{B11}
 \end{minipage}
 \begin{minipage}[b]{0.24\linewidth}
 \vspace{0pt}
 \centering
 \includegraphics[width=0.8\linewidth]{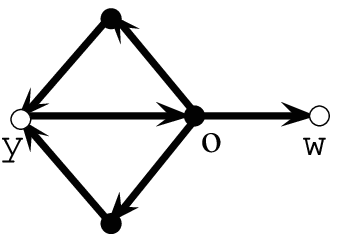}
 \caption{}\label{B12}
 \end{minipage}
% \begin{minipage}[b]{0.24\linewidth}
% \vspace{0pt}
% \centering
% \includegraphics[width=0.7\linewidth]{outwarddiamond.eps}
% \caption{}\label{B13}
% \end{minipage}
 \begin{minipage}[b]{0.24\linewidth}
 \centering
 \includegraphics[width=0.7\linewidth]{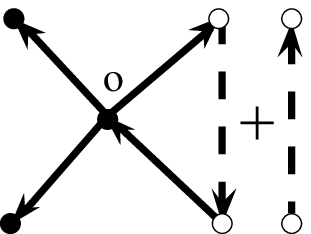}
 \caption{}\label{B2}
 \end{minipage}
\end{figure}
\noindent For Figure \ref{B11}, replace the neighborhood with the
one in Figure \ref{B11s}. According to Corollary \ref{sandwichlemma} and
\textit{Remark} \ref{rmk4}, this replacement is reversible. Denote
the new graph as $G'$ and the original graph as $G$. $G'$ has one
less node of degree 4. For Figure \ref{B12}, lemmas \ref{lma3out} and \ref{ywcomp} show
that it has a reversible replacement.
\begin{lma} If $y,w$ are not connected by an edge, Figure \ref{diaspike3} has only one possible
decomposition, shown in Figure \ref{diaspikedecom}\label{lma3out}
\end{lma}
\begin{figure}
 \centering
 \begin{minipage}[b]{0.3\linewidth}
 \centering
 \includegraphics[width=0.8\linewidth]{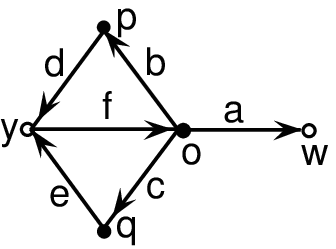}
 \caption{}\label{diaspike3}
 \end{minipage}
 \begin{minipage}[b]{0.3\linewidth}
 \centering
 \includegraphics[angle=180,origin=c,width=0.4\linewidth]{diamond.eps}
 \raisebox{1.2em}{+}
 \includegraphics[angle=-90,origin=c,width=0.4\linewidth]{spike.eps}
 \caption{}\label{diaspikedecom}
 \end{minipage}
 \begin{minipage}[b]{0.3\linewidth}
 \centering
 \includegraphics[width=1\linewidth]{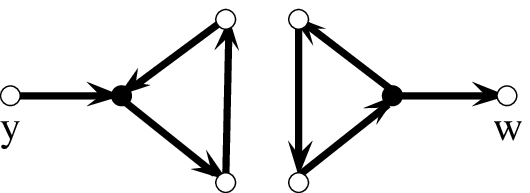}
 \caption{}\label{replacediaspike}
 \end{minipage}
\end{figure}
\begin{figure}
\centering
 \begin{minipage}[b]{0.4\linewidth}
 \centering
 \includegraphics[width=0.5\linewidth]{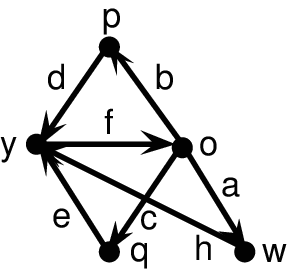}
 \caption{}\label{ywconnected}
 \end{minipage}
\end{figure}
\begin{proof}
Consider edge $a$. We claim that it comes from a spike block. We only need to rule out all other possibilities\\
Suppose $a$ comes from a fork, the other edge of the same fork can not be annihilated since
it has a black endpoint. Hence it must be edge $b$ or $c$. Assume
it is $b$, then degree of $b$ must be one. This is a contradiction.\\
Suppose $a$ comes from a square. Since the degree of $o$ is 4, $o$
must be the center of the square, which means edges $b,c,f$ are
contained in the same square block. This is a contradiction since $o$ must be incident to at least two outward edges and two inward edges.\\
Suppose $a$ is contained in a diamond. The degree of node $o$
suggests that $o$ is a white node in the diamond block containing $a$. Since the boundary edges of a diamond can not be
annihilated, two of $a,b,c,f$ must be boundary edges. Judging by the directions, the boundary edges can only be $\{a,b\}$, $\{a,c\}$
or $\{b,c\}$. If $\{a,b\}$ are two boundary edges, then $d$ must be contained in the
same diamond. This means node $w$ must be connected to node $y$.
This contradicts our assumption. The situation is similar if $\{a,c\}$
are two boundary edges. If $\{b,c\}$ are two boundary edges of
the diamond. Then $a$ is the mid-edge of the same diamond block. Therefore
nodes $p,q$ must be connected with node $w$, and they must be black. However
edge $d,e$ are incident to them. This is a contradiction.\\
Suppose $a$ comes from a triangular block. If this triangle does not contain edge $f$, the
other edge of the same triangle which is incident to $o$ must be
annihilated by another edge, denoted as $h$. So $b,c,f,h$ come from
the same block. It must be a square. However, none of the edges in a
square can be annihilated. This contradicts to the fact that $h$ is
annihilated. If the triangle contains $f$, then $b,c$ must come from
the same block. It must be a fork or a diamond. If it is the latter,
the mid-edge must be annihilated. But $o$ is already a black node
once the triangle and diamond are glued together. Contradiction. If
it is a fork, degree
$p$ must be 1. This is again a contradiction.\\
To sum up, $a$ must be a single spike.\\
Now $b,c,f$ come from the same block. This forces the block to be a diamond.
\end{proof}
\begin{lma}\label{ywcomp} If $w$ is connected to $y$ in Figure \ref{diaspike3},
then the decomposable graph must have a disjoint connected components shown in Figure \ref{ywconnected}.
\end{lma}
\begin{proof} According to the previous lemma, there are two
possibilities. Either\linebreak $b,c,d,e,f$ form a diamond and $a,h$
come from two spikes, or, $a,b,d,f,h$ form a diamond and $e,c$ come from two
spikes. In either case, the neighborhood is a disjoint connected component. Figure
\ref{ywconnected} illustrates the first case. To see the second case,
one only needs to change the labeling of the edges in Figure \ref{ywconnected}.
\end{proof}
%\begin{rmk}
%In the case of Figure \ref{B13}, we obtain the similar results.
%\end{rmk}
\begin{rmk}
The replacement of Figure \ref{diaspikedecom} $\leftrightarrow$ Figure \ref{replacediaspike} is reversible.
\end{rmk}

%==================Not necessary==================
%\begin{figure}[bth]
%\centering
%\end{figure}
%\noindent If $y$ is not connected to $w$, replace Figure \ref{B12},\ref{B13} with Figure
%\ref{replacediaspike}. Suppose the new graph is indecomposable. After
%removing the neighborhood from Figure \ref{replacediaspike},  the remaining graph is still
%undecomposable. Thus, if we retrieve the original graph by gluing back the neighborhood from
%Figure \ref{B12} or \ref{B13} in place of the one in Figure \ref{replacediaspike}, the resulting graph
%is undecomposable according to Lemma \ref{lma3out}. Thus, the
%replacement is reversible.\\\\
%==================begin again=======================
\noindent Assume now the node incident to the edge directed inwards has
degree one.
\begin{enumerate}
\item Suppose the inward edge comes from a spike. We show that the remaining
three edges can not come from one block, and this contradicts decomposability.
Indeed, there is no block that contains node of degree 3 that is with three outward edges.
\item Suppose the inward edge comes from a fork. Since degree of $o$ is 4, it must be the white node in the
fork. Hence one of the remaining edges is contained in the same
fork. However their directions are inconsistent with a fork. This is a
contradiction.
\item The inward edge can not be obtained from a
diamond since every node in a diamond has degree at least 2.
\item The same argument shows that the inward edge doesn't come from a square.
\item If the inward edge is obtained from a triangle, then by arguments as in Lemma
\ref{lma3out}, the triangle must contain one of the remaining
outward edges. The only possible decomposition is shown in Figure \ref{B2}. The dashed edge
can only be annihilated by a spike. Otherwise, the degree of the node will be greater than one.
In this case, the neighborhood is a disjoint connected component.
\end{enumerate}

\subsection{Two outward edges $+$ two inward edges}
Here we will distinguish cases by the number of boundary nodes of degree at least 2. Denote the number of such nodes by $n$.
For example, if $n=0$, it
is a 4-star.
\subsubsection{n=0}
The neighborhood can be constructed from gluing two forks, as shown in Figure \ref{2o2in}. Also, it can be constructed from
gluing two triangles, each has one edge annihilated. It must
be a disjoint connected component, otherwise, $G$ is undecomposable.
\subsubsection{n=1}\label{node1}
Without loss of generality, assume the node \textbf{1} incident to an outward
edge has degree at least 2. (Figure \ref{2o2in_1}).
\begin{figure}[tb]
\centering
\begin{minipage}[t]{0.24\linewidth}
\centering
\includegraphics{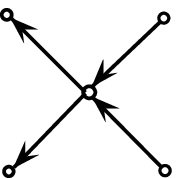}
\caption{}\label{2o2in}
\end{minipage}
\begin{minipage}[t]{0.24\linewidth}
\centering
\includegraphics{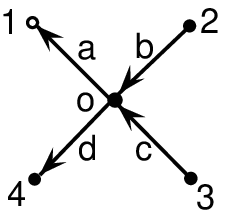}
\caption{}\label{2o2in_1}
\end{minipage}
\begin{minipage}[t]{0.24\linewidth}
\centering
\includegraphics{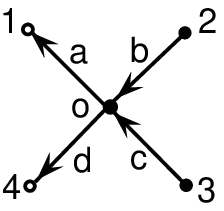}
\caption{}\label{2o2in_2}
\end{minipage}
\begin{minipage}[t]{0.24\linewidth}
\centering
\includegraphics{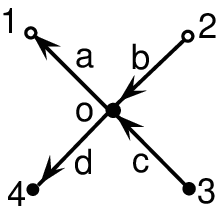}
\caption{}\label{2o2in_2_1}
\end{minipage}
\end{figure}
\begin{enumerate}
\item Edge $a$ does not come from a fork since the degrees of
both nodes~\textbf{1} and $o$ are at least 2.
\item Suppose $a$ comes from a diamond. Since degrees of nodes \textbf{2,3,4} are all
1, they can not be contained in the same diamond. So node $o$ is
a white node of the diamond before attaching edge $b,c,d$. Hence at least one
boundary edge in the diamond must be annihilated, which is impossible.
\item Suppose $a$ comes from a square. If $o$ is the central node of the square, edges
$b,c,d$ must be contained in the same square. Hence the remaining 4
nodes must be corner nodes. Thus, they all have degree 3. This
is a contradiction since only node 1 has degree more
than 1. So $o$ is a corner node of the square. But then
the degree of node $o$ must be 3. This contradicts the fact that
degree of $o$ is 4.
\item If $a$ comes from a spike block, $b,c,d$ must come from the same block, which must be a diamond. Hence,
edge $d$ is the mid-edge. However the degree of node \textbf{3} is
1, this is impossible since no boundary edge in a diamond can be
annihilated.
\item Assume that $a$ comes from a triangle $\triangle$. If the other edge of $\triangle$ incident to $o$ is not
$b$ or $c$, that edge must be annihilated by another one denoted as $e$, as shown in Figure \ref{2o2inillu}.
Thus, $b,c,d,e$ come from the same block. It can only be a square.
But the degrees of nodes \textbf{2,3,4} are all 1, which is
impossible for nodes in a square block. So either $b$ or $c$ is contained in the
same triangle.  Assume that it is $b$. Notice that node \textbf{2}
has degree 1. So the edge in $\triangle$ that connects node
\textbf{1} and \textbf{2} is annihilated by another edge, denoted as $f$. If $f$ comes from a
spike, the degree of node~\textbf{1} must be 1 after gluing. This is
a contradiction. If $f$ comes from a triangle or a diamond, the
degree of node \textbf{2} has degree at least 2 after gluing.
This is also a contradiction.
\end{enumerate}
To conclude, when $n=1$, the graph is undecomposable.
\begin{figure}
\centering
\begin{minipage}[b]{0.4\linewidth}
\centering
\includegraphics[width=0.5\linewidth]{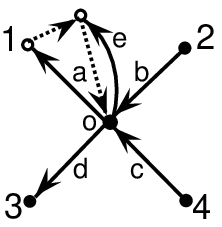}
\caption{}\label{2o2inillu}
\end{minipage}
\end{figure}
\subsubsection{n=2}
In this case, only two boundary nodes have degree at least 2.\\\\
{\bf Case 1.} Assume that the edges incident to the boundary nodes of
degree at least 2 have the same direction. Without loss of generality we assume that
both are directed outwards. (nodes \textbf{1} and \textbf{4} in Figure \ref{2o2in_2} have degree at least 2.) First, suppose either
$a$ or $d$ is a single spike, the remaining three edges must come from
the same blocks, which can only be a diamond or a square. However the
degrees of nodes \textbf{2,3} are both 1. This is impossible. Second, neither
of the edges $a$ or $d$ can be obtained from a fork since both of its
two endpoints have degree at least 2. Third, Suppose $a$ comes
from a diamond. Then $b,c$ must also be contained
in the diamond. In this case, nodes \textbf{1} and \textbf{2} must
be connected. This means the degree of node \textbf{2} is at least
2, which leads to a contradiction. Next, suppose $a$ or $d$ comes from a
square, then all four edges must be contained in the same square.
However, the degrees of node \textbf{2,3} are both one. This is again
a contradiction. Last of all, assume $a,b$ come from the same triangular block and $c,d$ come from another
triangular block. Since node \textbf{2} has degree
1, the third edge in the triangle containing edges $a,b$ is annihilated, as discussed in
the case when $n=2$, this is a contradiction. So in this case, the
graph is not
decomposable.\\\\
{\bf Case 2.} $o$ is connected to the boundary nodes of degree at least 2 by two edges. Denote the edge directed inwards by $a$ and the one directed outwards by $b$.(Figure \ref{2o2in_2_1}).
\begin{enumerate}
\item For the same reason as in Case 1, neither $a$ nor $b$ comes from
a spike block.
\item Neither $a$ nor $b$ comes from a fork since both endpoints have degrees at least 2.
\item Suppose $a$ comes from a diamond. Since the
degree of $o$ is 4, it must be a white node in the diamond. Since no
boundary edge in a diamond can be annihilated, $b,c$ must be
boundary edges in the same diamond. Then, nodes \textbf{1} and
\textbf{3} must be connected. But degree of node \textbf{3} is 1 and
the boundary edge can not be annihilated. This is a contradictions.
\item Suppose $a$ comes from a square block. Since degree of $o$ is 4, it
must be the central node of the square. Since none of the edges in a
square can be annihilated, $a,b,c,d$ must all be contained in the
same square. But the degrees of node \textbf{3,4} are one.
Contradiction.
\item Assume $a$ comes from a triangle $\triangle$. Suppose this triangle does not contain $b$ or $c$. The other edge in this triangle that is incident to $o$ must annihilated by an edge, denoted by $e$. Hence edges $b,c,d,e$ must be contained in a square block. However none of the edges in a square block can be annihilated. Therefore, the triangle must contain either $b$ or $c$. Using similar arguments as in section
\ref{node1}, $c$ is not contained in $\triangle$. So $a,b$ are contained in $\triangle$. Then we replace the neighborhood with the one in Figure
\ref{startri}. The replacement operation is reversible by Corollary \ref{sandwichlemma}.
\end{enumerate}
\begin{figure}[tbc]
\centering
 \begin{minipage}[b]{0.4\linewidth}
 \centering
 \includegraphics[width=0.4\linewidth]{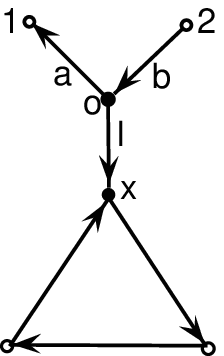}
 \caption{Case 2 replacement}\label{startri}
 \end{minipage}
 \begin{minipage}[b]{0.4\linewidth}
 \centering
 \includegraphics[width=0.8\linewidth]{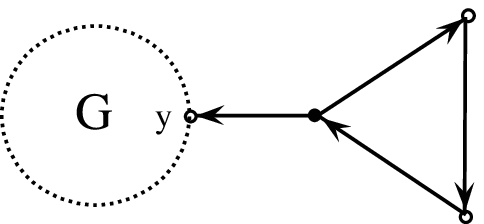}
 \caption{G'}\label{G'}
\end{minipage}
\end{figure}
\subsubsection{n=3}
Without loss of generality, assume node \textbf{1} is incident to the edge
directed outwards (denoted as $a$), and it has degree one, see
Figure \ref{2o2in_3}.
\begin{figure}[bt]
 \centering
 \begin{minipage}[b]{0.24\linewidth}
 \centering
 \includegraphics[width=0.8\linewidth]{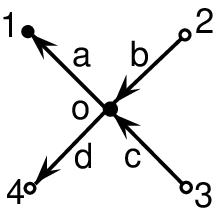}
 \caption{}\label{2o2in_3}
 \end{minipage}
 \begin{minipage}[b]{0.24\linewidth}
 \centering
 \includegraphics[width=1\linewidth]{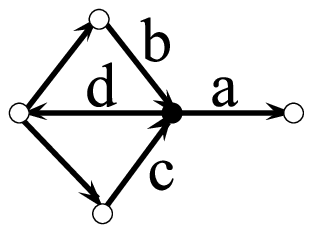}
 \caption{}\label{ds6}
 \end{minipage}
% \begin{minipage}[b]{0.24\linewidth}
% \centering
% \includegraphics[width=0.55\linewidth]{diaspikeillu.eps}
% \caption{}\label{DSDillu}
% \end{minipage}
\end{figure}
\begin{enumerate}
\item Suppose that $a$ comes from a single spike. The remaining edges $b,c,d$ must come
from the same block. The only possible situation is that they come
from a diamond (Figure \ref{ds6}). Since
deg(\textbf{1})=1, nodes \textbf{1,4} are not connected. Lemma \ref{decomdiaspike} below shows
that this neighborhood can be replaced by the one in Figure \ref{G'}.
This replacement is reversible according to Lemma
\ref{trianglespike}.
\item Suppose that $a$ comes from a fork. Since deg($o$)$=4$, $o$
must be the white node in the fork. Then $d$ is also contained in
the same fork. Hence, node \textbf{4} must have degree 1. This
contradicts the fact that the degree of node \textbf{4} at least 2. So $a$ does not come from a fork.
\item Assume $a$ comes from a triangle. According to the argument in section
\ref{node3o1in}, this triangle must contain
edge $b$ or $c$  Assume that the triangle contains $a,b$. Since the
degree of node \textbf{1} is one and the degree of node \textbf{2}
is at least two, we obtain a contradiction by arguments from section
\ref{node1}.
\item Suppose that $a$ comes from a diamond, then the degree of node \textbf{1}
must be at least 2, which is a contradiction.
\item Suppose that $a$ comes from a square block. This block must also contain edges $b,c,d$ since
none of the edges in a square can be annihilated. Moreover, $o$ is
the central block of the square. This means node \textbf{1} must have
degree 3. This is a contradiction.
\end{enumerate}
%=======================saved============
\subsubsection{n=4}
\noindent In this section, we assume all four boundary nodes have degree at least 2. (Figure \ref{2o2in_4}).We focus our discussion on edge $a$. By the symmetry of the neighborhood, we can carry the same argument to any of edges $b,c,d$.
\begin{figure}[btch]
\centering
 \begin{minipage}[b]{0.23\linewidth}
 \centering
 \includegraphics[width=0.75\linewidth]{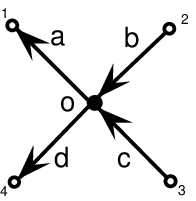}
 \caption{}\label{2o2in_4}
 \end{minipage}
\end{figure}
\begin{enumerate}
\item Edge $a$ does not come from a fork since both of its endpoints have
degrees at least 2.
\begin{figure}[btch]
 \centering
  \begin{minipage}[c]{0.3\linewidth}
 \centering
 \includegraphics[width=0.6\linewidth]{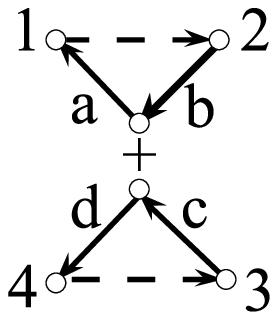}
 \caption{}\label{twotri}
 \end{minipage}
 \begin{minipage}[c]{0.3\linewidth}
 \centering
 \includegraphics[width=0.5\linewidth]{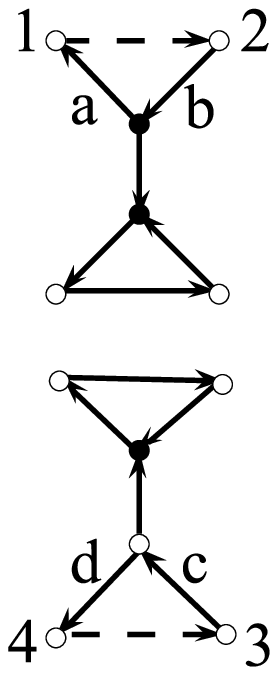}
 \caption{}\label{twotrisimp}
 \end{minipage}
 \begin{minipage}[c]{0.3\linewidth}
 \centering
 \includegraphics[width=0.5\linewidth]{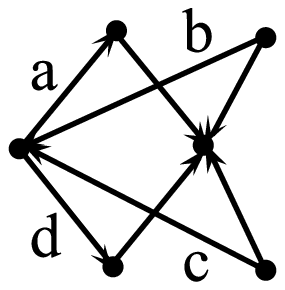}
 \caption{}\label{doublediamond2}
 \end{minipage}
\end{figure}
\item Assume that $a$ comes from a triangle. Similar to argument in
section \ref{node3o1in}, the triangle must contain $b$ or $c$.
Assume $b$ is contained in this triangle. Then $c$ and $d$ must
come from the same block.
\begin{itemize}
\item If this block is a diamond, judging by their directions, one of edges $c,d$ (assume it is $d$) must be the
mid-edge. Thus, besides $c$, there is another boundary edge incident to $o$
that comes from the same diamond. Hence, the degree of
$o$ is at least 5. This contradicts our assumption.\\
\item Assume $c,d$ come from a triangle, as shown in Figure \ref{twotri}. We replace the
neighborhood by the one in Figure \ref{twotrisimp}.
\end{itemize}
\begin{rmk} Notice that in order to perform replacement, it is necessary to determine
whether $a,b$ or $a,c$ are in the same triangle. This will be
discussed later.
\end{rmk}
\item Suppose that $a$ comes from a diamond. Since the degree of $o$ is 4, it
must be a white node of the diamond. Judging by the directions of
edges, there are three cases.
\begin{itemize}
\item $a$ is the mid-edge and
$b,c$ are the boundary edges. In this case, we obtain a neighborhood as shown in Figure \ref{ds4}. We will discuss it later in this section.
\item $a,d$ are the boundary edges and $b$ or $c$
is the mid-edge. We get the neighborhood shown in Figure
\ref{ds5}.
\begin{figure}[btch]
\centering
\begin{minipage}[c]{0.25\linewidth}
\centering
\includegraphics{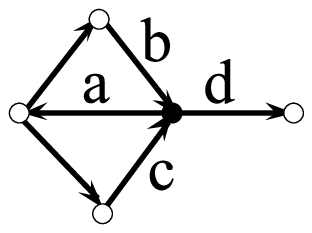}
\caption{}\label{ds4}
\end{minipage}
\begin{minipage}[c]{0.25\linewidth}
\centering
\includegraphics{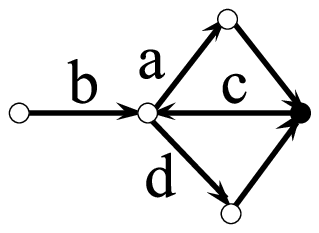}
\caption{}\label{ds5}
\end{minipage}
\end{figure}
\item $a,d$ are the boundary edges, and the
mid-edge is annihilated by another edge $e$. So $b,c,e$ come from
the same block. It must be a diamond with mid-edge $e$, see Figure \ref{doublediamond2}. In this case, the neighborhood is a disjoint
connected component.
\end{itemize}
\item If $a$ comes from a spike, $b,c,d$ must come from the same block.
Hence, this block must be a diamond, see Figure \ref{ds6}.
\item Suppose $a$ comes from a square, then $a,b,c,f$ must all be contained
in the same square. Thus, the neighborhood is the square itself.
Since the degree of $o$ is 4, the neighborhood is a disjoint
connected component.
%\item Suppose $a$ comes from a diamond. We have two choices for
%decomposition depending on how the graph look like Figure
%\ref{doublediamond} and Figure \ref{doublediamond2}. For Figure
%\ref{doublediamond}, we replace the neighborhood with the one in
%\ref{replacedd}. For Figure \ref{doublediamond2}, the neighborhood
%is a disjoint connected component.
\end{enumerate}
%\begin{figure}[tb]
% \centering
% \begin{minipage}[b]{0.3\linewidth}
% \centering
% \includegraphics[width=0.4\linewidth]{doublediamond.eps}
% \caption{}\label{doublediamond}
% \end{minipage}
% \begin{minipage}[b]{0.3\linewidth}
% \centering
% \includegraphics[width=0.25\linewidth]{replacedoublediamond.eps}
% \caption{}\label{replacedd}
% \end{minipage}
%\end{figure}
\begin{figure}
 \centering
 \begin{minipage}[b]{0.25\linewidth}
 \centering
 \includegraphics[width=1\linewidth]{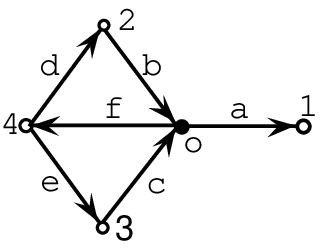}
 \caption{}\label{diaspike3plus}
 \end{minipage}
 \begin{minipage}[b]{0.25\linewidth}
 \centering
 %auto-ignore
% Generated with LaTeXDraw 2.0.8
% Wed Jul 07 17:48:32 PDT 2010
% \usepackage[usenames,dvipsnames]{pstricks}
% \usepackage{epsfig}
% \usepackage{pst-grad} % For gradients
% \usepackage{pst-plot} % For axes
\scalebox{1.2} % Change this value to rescale the drawing.
{
\begin{pspicture}(0,-0.3783375)(3.3576136,1.6141753)
\psline[linewidth=0.03cm,arrowsize=0.05291667cm 2.0,arrowlength=1.4,arrowinset=0.4]{<-}(1.4776136,1.3416625)(0.7976135,0.5616625)
\psline[linewidth=0.03cm,arrowsize=0.05291667cm 2.0,arrowlength=1.4,arrowinset=0.4]{<-}(2.2176135,0.5416625)(1.4776136,1.3416625)
\psline[linewidth=0.03cm,arrowsize=0.05291667cm 2.0,arrowlength=1.4,arrowinset=0.4]{<-}(1.4776136,-0.2583375)(0.7776135,0.5616625)
\psline[linewidth=0.03cm,arrowsize=0.05291667cm 2.0,arrowlength=1.4,arrowinset=0.4]{<-}(2.1776135,0.5416625)(1.4776136,-0.2583375)
\psline[linewidth=0.03cm,arrowsize=0.05291667cm 2.0,arrowlength=1.4,arrowinset=0.4]{<-}(0.8376135,0.5416625)(2.1376135,0.5416625)
\psline[linewidth=0.03cm,arrowsize=0.05291667cm 2.0,arrowlength=1.4,arrowinset=0.4]{<-}(3.2576134,0.5416625)(2.1776135,0.5416625)
\psdots[dotsize=0.12](2.1576135,0.5416625)
\psdots[dotsize=0.12,fillstyle=solid,dotstyle=o](3.2776136,0.5416625)
\psdots[dotsize=0.12,fillstyle=solid,dotstyle=o](1.4576136,-0.2983375)
\psdots[dotsize=0.12,fillstyle=solid,dotstyle=o](0.7776135,0.5416625)
\psdots[dotsize=0.12,fillstyle=solid,dotstyle=o](1.4576136,1.3416625)
\psline[linewidth=0.03cm,linestyle=dashed,dash=0.16cm 0.16cm,arrowsize=0.05291667cm 2.0,arrowlength=1.4,arrowinset=0.4]{<-}(1.5176135,1.3816625)(3.3176136,0.5216625)
\rput{23.094828}(0.1537633,-0.75258243){\psarc[linewidth=0.03,linestyle=dashed,dash=0.16cm 0.16cm,arrowsize=0.05291667cm 2.0,arrowlength=1.4,arrowinset=0.4]{<-}(1.9186096,0.0){1.4622312}{0.0}{81.70286}}
\usefont{T1}{ptm}{m}{n}
\rput(2.6471448,1.4916625){$\lambda$}
\usefont{T1}{ptm}{m}{n}
\rput(2.5907385,0.3916625){a}
\usefont{T1}{ptm}{m}{n}
\rput(1.6018323,0.9916625){b}
\usefont{T1}{ptm}{m}{n}
\rput(1.4290198,0.6916625){f}
\usefont{T1}{ptm}{m}{n}
\rput(0.9021448,0.9716625){d}
\usefont{T1}{ptm}{m}{n}
\rput(1.6055822,0.1316625){c}
\usefont{T1}{ptm}{m}{n}
\rput(0.9257385,0.1316625){e}
\end{pspicture}
} 
 \caption{}\label{newillu}
 \end{minipage}
 \begin{minipage}[b]{0.3\linewidth}
 \centering
 \raisebox{1.5em}{\includegraphics[origin=c,width=0.35\linewidth]{diamond.eps}
 \raisebox{2.4em}{\small{+}}
 \includegraphics[angle=-90,origin=c,width=0.35\linewidth]{spike.eps}}
% \vspace{2em}
 \caption{}\label{DSD}
 \end{minipage}
\end{figure}

\vspace{1em}\noindent Note that Figure \ref{ds4},\ref{ds5},\ref{ds6} represent the same neighborhood except for edge labeling. For the sake of convenience, we relabel the edges as in Figure \ref{diaspike3plus}. Note that the degree of node
\textbf{1} is at least 2. If \textbf{1} is not connected to \textbf{4}, the
only possible decomposition is the one shown in Figure \ref{DSD} (See Lemma
\ref{decomdiaspike}). We apply the replacement as in Figure
\ref{replacediaspike}. If nodes \textbf{1,4} are connected by an edge directed
from \textbf{1} to \textbf{4}, Lemma \ref{decomdiaspikeconnect} shows that there exists a
decomposition as in Figure \ref{decom}. Thus, we can apply the replacement as in Figure \ref{G''}. The following lemmas show that our choices of  replacements for the neighborhood in Figure \ref{diaspike3plus} are reversible.
\begin{lma}In Figure \ref{diaspike3plus}, assume $4$ is neither connected to \textbf{1} nor coincide with \textbf{1}. Nodes \textbf{1,2}, nodes \textbf{1,3} are disconnected. If the graph
$G$ is decomposable, then the neighborhood of $o$ can be decomposed
as in Figure \ref{DSD}. \label{decomdiaspike}
\end{lma}
\begin{proof}
\begin{enumerate}
\item Suppose that $a$ comes from a fork. Since the degree of $o$ is 4,
$o$ must be the white node in the fork. Thus, $f$ is contained in
the same fork. Then node \textbf{4} must have degree 1. This is a contradiction.
\item Suppose that $a$ comes from a triangle, denoted as $\triangle$. Then there are two cases:\\
{\bf Case 1}: $\triangle$ contains neither $b$ nor $c$; \\
{\bf Case 2}: $\triangle$ contains either $b$ or $c$\\
In case 1, consider node $o$ in Figure \ref{diaspike3plus}. The
other edge in $\triangle$ that is incident to $o$ is
annihilated by another edge, denoted as $e$. Hence $b,c,f,e$ come from the
same block, which can only be a square block. However, none of the edges
in a square can be annihilated. Therefore, case 1 is impossible. In
case 2, assume that $\triangle$ contains $b$. The third edge in
$\triangle$ is annihilated by another edge,
denoted as $\lambda$. (See Figure \ref{newillu}). $\lambda$ and $d$
must come from the same block. It can only be a diamond or a
triangle. If it is a diamond, $\lambda$ must be the mid-edge. So
node \textbf{4} is black. However, edges $f,e$ need to be glued to \textbf{4}. This
is a contradiction. So both $b,\lambda$ belong to a triangle.
Since \textbf{4} is not connected to \textbf{1}, the edge $\overline{14}$ in this
triangle must be annihilated by another edge $h$. So $h,f,e$ come
from the same block, which must be a diamond and $h$ is the mid-edge.
This means that nodes $o$ and \textbf{1} are connected by a boundary edge of this
diamond. Thus, degree of $o$ is at least 5. This contradicts to the assumption
that deg($o$)=4.
\item Suppose that $a$ comes from a diamond. Since the degree of $o$ is 4, it
must be a white node in the diamond. Since the boundary edges can
not be annihilated. Judging by the directions of the edges, there are
only two possible cases:
\begin{itemize}
\item $a$ is the mid-edge and $b,c$
are two boundary edges of the diamond.
\item $a,f$ are the boundary edges and one
of $b,c$ is mid-edge.
\end{itemize}
In either cases, \textbf{1,2} must
be connected by a boundary edge and it can not be annihilated. This
is a contradiction.
\item Suppose that $a$ comes from a square, then $a,b,c,f$ must all be
contained in the same square. Thus, the neighborhood is the square. Moreover, nodes \textbf{1,2} must be connected. This is a contradiction.
\item Suppose edge $a$ comes from a spike. Then $b,c,f$
come from the same block. This forces the block to be a diamond. See Figure \ref{DSD}
\end{enumerate}
\end{proof}
\begin{lma}\label{decomdiaspikeconnect} In Figure \ref{diaspike3plus}, assume that \textbf{4} is connected to \textbf{1} by an
edge directed from \textbf{1} to \textbf{4}. If the
graph is decomposable, nodes \textbf{1,2} and node \textbf{1,3} are disconnected. Then the degree of \textbf{4} is 4 and the degree of
\textbf{1} is 2. In this case, there is a decomposition as in Figure
\ref{decom}. Also, it is also possible to simplify the original
graph $G$ to $G'$ (Figure \ref{G''}). $G$ is decomposable if and
only if $G'$ is decomposable. If the degree of node \textbf{3} is 2, there
is an alternative decomposition as in Figure \ref{special}. In this case,
$G$ is a disjoint connected component.
\end{lma}
\begin{figure}[bt]
 \centering
 \begin{minipage}[b]{0.3\linewidth}
 \centering
 \includegraphics[width=0.9\linewidth]{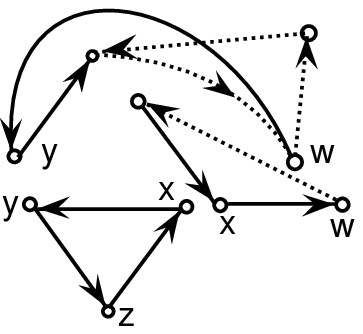}
% \vspace{-0.5em}
 \caption{}\label{decom}
 \end{minipage}
 \begin{minipage}[b]{0.3\linewidth}
 \centering
 \includegraphics[width=0.5\linewidth]{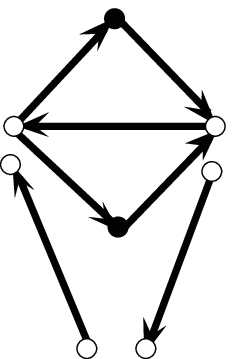}
 \caption{}\label{special}
 \end{minipage}
 \begin{minipage}[b]{0.3\linewidth}
 \centering
 \includegraphics[width=0.5\linewidth]{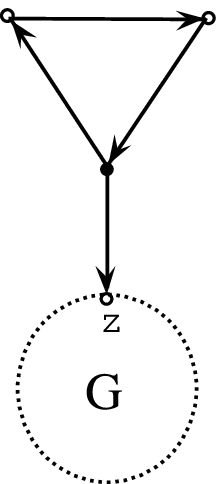}
% \vspace{-1em}
 \caption{G'}\label{G''}
 \end{minipage}
\end{figure}
\begin{proof}
The argument differs from the previous one only in the place
when $a$ is assumed to come from a triangle. Notice that \textbf{4} is
connected to \textbf{1}. If $a,b$ comes from a triangular block $\triangle$,
the edge $\overline{41}$ must come from another block. This block can be a
triangle or a diamond. Thus, $e,f$ must both come from the other
block, which can not be a diamond since this will force the degree of node
\textbf{4} to be 5. Recall that we already simplified all nodes of decomposable
graph so that the degree of any node does not exceed 4. Thus, this block containing
$e,f$ must be a triangle. The corresponding decomposition is shown in Figure \ref{decom}).
In this case, if degree of node \textbf{3} is at least 3, there is another edge incident to it. The neighborhood can
be replaced by the one in Figure \ref{G''}. It is trivial
that if $G$ is decomposable, so is $G'$. The converse statement follows
from Lemma \ref{trianglespike}.
\end{proof}
\begin{rmk} Note that we have found all possible decomposition of the neighborhoods of
nodes with degree 4. Except some cases when the neighborhood is a disjoint connected component,
we want to identify which decompositions the considered neighborhood can have. To be more specific, to determine decomposition of the neighborhood of a node $o$ we want to
use only the information that can be directly derived from the graph:
\begin{itemize}
\item How is node $o$ connected to the boundary nodes? We want to check the direction of the edges connecting node $o$ and its boundary nodes.
\item How are the boundary nodes connected to each other? We want to check if and how some of the boundary nodes are connected to each other.
\item If necessary, we want to check if there is any other node that is connected to the boundary nodes, and how are they connected.
\end{itemize}
This method will be discussed in detail in the next section.
\end{rmk}
\begin{rmk} If node \textbf{4} coincides with node \textbf{1}, we have a neighborhood as in Figure \ref{dd}.
In this case, we need to examine nodes $p,q$.
\begin{itemize}
\item If both nodes have degree
two, then there are two possible decomposition. Namely, a diamond plus
a spike or two triangles.
\item If at least one of $p,q$ has degree more than two,
then it must come from gluing two triangles.
\end{itemize}
Figure \ref{dd} shows the decomposable neighborhood. All cases other the the above two give an undecomposable graph.
\end{rmk}
\begin{figure}
\centering \begin{minipage}[b]{0.4\linewidth} \centering
\includegraphics[width=0.5\linewidth]{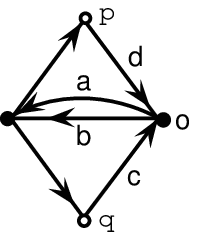}
\caption{}\label{dd}
\end{minipage}
\end{figure}
\section{Identification when n=4}
\vspace{1em}In the previous section, we have found all possible neighborhoods and possible decomposition of node of degree 4. In order to perform proper replacement, we need to identify the neighborhood by examining the boundary nodes of node $o$. For example, in the situation when the neighborhood may come from two triangles, in order to choose proper replacement, we must determine whether $a,b$ or $a,c$ are in the
same triangle. Also, in some other cases, we need to determine which one of the
four edges comes from a spike, the remaining three edges then come from a diamond.\\\\
To determine decomposition, we must consider connectivity between the boundary nodes.
First of all, we need to consider decompositions depending on how
nodes \textbf{3,4} are connected to node \textbf{1}.\\
\subsection{Node \textbf{1} is Connected to Node \textbf{4} and \textbf{3}}
Assume nodes \textbf{1,4} are connected by an edge denoted by $\lambda$ and nodes
\textbf{1,3} are connected by an edge denoted by $\gamma$. Let us consider directions of
$a,d,\lambda$ and $a,c,\gamma$. More exactly, we check if $\lambda$ is directed from node
\textbf{1} to node \textbf{4} and
if $\gamma$ is directed from node \textbf{1} to \textbf{3}. (See Figure \ref{1to23})\\
\begin{figure}[btch]
\centering
\includegraphics[width=0.21\linewidth]{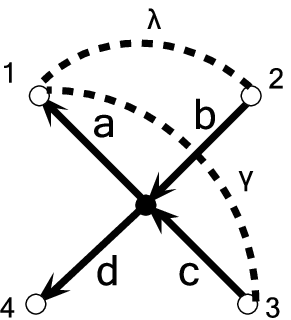}
\caption{}\label{1to23}
\end{figure}
Suppose $\lambda$ is directed from node
\textbf{2} to node \textbf{1} and
$\gamma$ is directed from node \textbf{3} to \textbf{1}, then neither
$a,b,\lambda$ nor $a,c,\gamma$ come from a triangle. Assume $a$
comes from a spike, then $b,c,d$ come from the same block which must
be a diamond. Hence, nodes \textbf{2,4} must be connected and node
\textbf{2} is a black node before $\lambda$ is attached. In this case,
node \textbf{1} must coincide with node \textbf{4}. But the directions of
$\lambda,\gamma$ are prescribed by the decomposition. Hence, the graph is
undecomposable. If $a$ comes from a diamond, the diamond must also
contain $b,c,\lambda,\gamma$. Again, their directions do not fit in a diamond block.
To conclude, if $a,b,\lambda$ or $a,c,\gamma$ can not
form a triangular block,
the graph is undecomposable.\\\\
Suppose $a,b,\lambda$ have the same direction setup as a triangular block, and
$a,c,\gamma$ don't. We claim that if the graph is decomposable, then
$a,b,\lambda$ must come from a triangular block. Suppose the contrary. Notice
that $a,c,\gamma$ do not come from a triangular block. Edge $a$ comes either
from a spike or from a diamond. In the first case, $b,c,d$ come from the same block which
must be a diamond, and node \textbf{2} is connected only to nodes \textbf{4} and $o$. Since node
\textbf{2} is connected to \textbf{1}, node
\textbf{1} must coincide with node \textbf{4}. But then the direction of
$\gamma$ does not match the direction of the corresponding edge in a diamond. This is a contradiction. If $a$ comes from a
diamond block, the block can contain $b,c,\lambda,\gamma$ or $b,d,\lambda$ or
$c,d,\gamma$.
But none of this cases has the directions that match with a diamond block. This again leads to a contradiction.\\\\
Suppose $\lambda=\overrightarrow{12}$and
$\gamma=\overrightarrow{13}$. If $a$
comes from a spike, then $\lambda,\gamma$ come from the same block.
This block can be a fork or a diamond. But the former is impossible since
the degree of node \textbf{2} is 2. If it is the latter, $b,c$ must
be boundary edges of this diamond. Thus, the mid-edge must connect
node \textbf{1} and $o$. This forces node \textbf{1} to coincide with node \textbf{4}, as shown in Figure \ref{dd}. Suppose that $a$ come from a diamond. There are two
possibilities. First, $a,b,d$ come from the same diamond. Second,
$a,b,c$ come from the same diamond. If it is the former, node
\textbf{1} is already black before $\gamma$ is glued. This is impossible.
Suppose it is the latter. Notice that node \textbf{2,3} and
\textbf{3,4} are disconnected unless nodes \textbf{1} coincide with node \textbf{4}.
\begin{lma}
Suppose there is an edge directed from node
\textbf{1} to node \textbf{2} and an edge directed from node \textbf{1}
to node \textbf{3}, nodes \textbf{2,3} and node \textbf{3,4} are
disconnected, both nodes \textbf{3} and \textbf{4} have degree 2. If the graph $G$ is decomposable,
then there is a decomposition of $G$ in which
$a,b,c,\lambda,\gamma$ come from the same diamond.
\end{lma}
\begin{proof} Suppose that it is false.
\begin{enumerate}
\item If $a$ comes from a spike, then $b,c,d$ come from the same
block which must be a diamond. Thus, node \textbf{3} is a black node of the diamond and
$\gamma$ can not be attached. Contradiction.
\item If $a$ comes from a triangle, the block could contain either edges
$b,\lambda$ or edges $c,\gamma$.\\
In the former case, edges $c,d$ come from the same triangle
$\triangle$. Since nodes \textbf{3,4} are disconnected the third
edge of $\triangle$ is annihilated by another edge, denoted as
$\tau$. Hence $\tau,\gamma$ come from the same triangle, and node
\textbf{1} is connected to \textbf{4}. If $G$ is decomposable, the neighborhood is as in Figure \ref{tritri}.\\
\noindent Moreover,
degree of node \textbf{1} is 4 and degree of node \textbf{4} is 2.
Notice that the degree of node \textbf{2} is 2, the neighborhood
is a disjoint connected component. In this case, it can also be
decomposed as a diamond containing
$a,b,c,\lambda,\gamma$ plus two spikes.\\
In the latter case,  the triangle which contains $a$ also contains $c,\gamma$. By the same argument as above,
if $G$ is decomposable, the neighborhood of $o$ is as in Figure \ref{tritri}. In this case, it's a disjoint connected component,
and it can be obtained by gluing two spikes $\overline{14},d$ to a
diamond that contains $a,b,c,\lambda,\gamma$.
\begin{figure}
\centering
\includegraphics[width=0.22\linewidth]{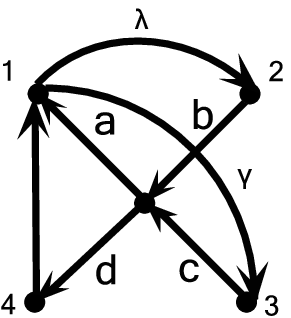}
\caption{}\label{tritri}
\end{figure}
\item Suppose $a$ comes from a diamond. According to the assumption,
$a,b,c,\lambda,\gamma$ do not come from the same diamond. Therefore, the
diamond containing $a$ must contain $b,d$ with $b$ as its mid-edge.
Then node \textbf{1} is black and $\gamma$ can not be attached.
Contradiction.\\
\item If $a$ comes from a square, it must contain
$b,c,d,\lambda,\gamma$. Moreover, nodes \textbf{2,4} and
\textbf{3,4} must be connected. This
contradicts our assumption.\\
\end{enumerate}
To conclude, under the given assumption, $a,b,c,\lambda,\gamma$ come
from the same diamond in one of the decomposition of $G$.
\end{proof}
\begin{rmk}If nodes \textbf{2,3} and nodes \textbf{3,4} are connected,
and degrees of nodes \textbf{2,3} are both two, node \textbf{4} must
coincide with node \textbf{1}. In this case, the neighborhood is
shown in the Figure \ref{dd}.
\end{rmk}
\begin{rmk}
In the above situation, if $G$ is decomposable, we may have more than one decompositions.
However, according to the proof of the lemma, there are more than one decomposition only when
the neighborhood is a disjoint connected component. We list all such disjoint connected components in Figure \ref{nonunique}. If the whole graph coincides with such a disjoint component from this list we know already all the possible decompositions and we don't need to do simplifying replacements. On the other hand, if a decomposable graph does not coincide with any of the graphs in Figure \ref{nonunique}, then the decomposition is unique.
\end{rmk}
\begin{rmk}
The lemma above explains that by examining the connectivity of nodes \textbf{2,3} and
nodes \textbf{3,4}, we can tell if $a$ comes from a
diamond. Moreover, if we can rule out the possibility
that $a$ comes from a diamond and node \textbf{1}$\neq$\textbf{4},
we can furthermore check if the neighborhood comes from a square.
In the following argument, assume we already rule out the possibility that $a,b,c,d$ comes from square, diamond or spike. This means, if the graph is decomposable, edges $a,b,c,d$ must come from two triangles.
\end{rmk}
\vspace{1em}\noindent In order to determine if $a,b$ or $a,c$
come from the same triangle, it is necessary to examine node
\textbf{4}.
\\\\
Assume node \textbf{4} is connected to both nodes \textbf{3} and
\textbf{2}, see Figure \ref{22od_1}. If nodes \textbf{3,4} are connected by an edge directed from node \textbf{3} to nodes \textbf{4}, relabel the indices of nodes as the following: \textbf{4} $\rightarrow$ \textbf{1}, \textbf{3} $\rightarrow$ \textbf{2}, \textbf{2} $\rightarrow$ \textbf{3} and \textbf{1} $\rightarrow$ \textbf{4}. Then apply the previous argument. It's similar if nodes \textbf{2,4} are connected by an edge directed from node \textbf{2} to nodes \textbf{4}. Hence, without loss of generality,
assume edge $\overline{24}$ is directed from node \textbf{4} to
\textbf{2} and edge $\overline{34}$ is directed from node \textbf{4}
to \textbf{3}. If there is no node (except for
node $o$) which is connected to any of the nodes
\textbf{1,2,3,4}, then it is a disjoint connected component.\\\\
\begin{figure}[btch]
\centering
 \begin{minipage}[b]{0.24\linewidth}
 \centering
 %auto-ignore
% Generated with LaTeXDraw 2.0.8
% Fri Jul 09 15:23:53 PDT 2010
% \usepackage[usenames,dvipsnames]{pstricks}
% \usepackage{epsfig}
% \usepackage{pst-grad} % For gradients
% \usepackage{pst-plot} % For axes
\scalebox{1.2} % Change this value to rescale the drawing.
{
\begin{pspicture}(1,-0.46249095)(3.3936412,1.7518841)
\psline[linewidth=0.04cm,arrowsize=0.05291667cm 2.0,arrowlength=1.4,arrowinset=0.4]{<-}(1.5086414,1.361884)(2.2886415,0.581884)
\psline[linewidth=0.04cm,arrowsize=0.05291667cm 2.0,arrowlength=1.4,arrowinset=0.4]{<-}(2.2886415,0.581884)(3.0886414,1.341884)
\psline[linewidth=0.04cm,arrowsize=0.05291667cm 2.0,arrowlength=1.4,arrowinset=0.4]{<-}(2.2886415,0.581884)(3.0886414,-0.21811596)
\psline[linewidth=0.04cm,arrowsize=0.05291667cm 2.0,arrowlength=1.4,arrowinset=0.4]{<-}(1.4886414,-0.23811595)(2.2886415,0.56188405)
\psline[linewidth=0.04cm,arrowsize=0.05291667cm 2.0,arrowlength=1.4,arrowinset=0.4]{<-}(3.0686414,1.361884)(1.5086414,1.361884)
\psline[linewidth=0.04cm,arrowsize=0.05291667cm 2.0,arrowlength=1.4,arrowinset=0.4]{<-}(3.0686414,-0.23811595)(1.5286413,-0.23811595)
\psdots[dotsize=0.12](2.2886415,0.56188405)
\psdots[dotsize=0.12](3.0886414,1.361884)
\psdots[dotsize=0.12](1.4686414,1.361884)
\psdots[dotsize=0.12](1.4686414,-0.25811595)
\psdots[dotsize=0.12](3.0886414,-0.25811595)
\rput{-84.55601}(1.5587314,1.5094286){\psarc[linewidth=0.04,arrowsize=0.05291667cm 2.0,arrowlength=1.4,arrowinset=0.4]{<-}(1.6094258,-0.10245822){1.4760472}{80.36318}{180.0}}
\rput{1.4405128}(-0.003158241,-0.0736943){\psarc[linewidth=0.04,arrowsize=0.05291667cm 2.0,arrowlength=1.4,arrowinset=0.4]{<-}(2.9294257,-0.16245823){1.4760472}{80.36318}{180.0}}
\usefont{T1}{ptm}{m}{n}
\rput(2.289579,1.571884){$\lambda$}
\usefont{T1}{ptm}{m}{n}
\rput(3.1428602,0.81188405){$\gamma$}
\usefont{T1}{ptm}{m}{n}
\rput(1.3155164,1.5318841){1}
\usefont{T1}{ptm}{m}{n}
\rput(3.267235,1.491884){2}
\usefont{T1}{ptm}{m}{n}
\rput(3.2162976,-0.30811596){3}
\usefont{T1}{ptm}{m}{n}
\rput(1.3295789,-0.28811595){4}
\end{pspicture}
}
 \caption{}\label{22od_1}
 \end{minipage}
\end{figure}
Suppose we can find a node $x\neq o$ that is connected to some of the nodes \textbf{1,2,3,4}.
We check if there is any node among \textbf{1,2,3,4} that is connected to $x$  Assume
$x$ is connected to only one of \textbf{1,2,3,4}. Without loss of generality,
assume $x$ is connected to \textbf{1} by an edge denoted as $\tau$. (Notice that $x$ may be connected to
nodes in the graph other than \textbf{1,2,3,4}). In this case, if edges $a,b$ come from
the same triangle, then edges $c,d$ come from another triangle,
denoted as $\triangle_1$. Moreover, $\tau,\gamma$ come from the same
block, which must be a triangle $\triangle_2$. Because $x$ is only
connected to one of nodes \textbf{1,2,3,4}, the third edge of
$\triangle_2$ is annihilated by another edge, denoted as $\eta$.
However, nothing can be attached to the node \textbf{3} after gluing $\triangle_2$ to
$\triangle_1$. Contradiction. Hence, edges $a,c$ come from the same
triangle and $b,d$ come from the same triangle. Therefore, edges
$\tau$ and $\lambda$
come from the same block. This is impossible.\\\\
Assume that there is a node $x\neq o$ that is connected to only two nodes of \textbf{1,2,3,4}.
Notice that $x$ may be connected to nodes
other than \textbf{1,2,3,4}. Up to a relabeling of indices, there are
two possible situations: Either $x$ is connected to nodes \textbf{1,4} or
nodes \textbf{1,3}. \\\\
First, suppose $x$ is connected to nodes \textbf{1,4}
by edges $\tau,\eta$ respectively.\\
If $a$ comes from a spike, then $\tau,\lambda,\gamma$ come from the
same block that is a diamond. Judging by the directions of $\lambda,\gamma$, $\tau$
must be the mid-edge. Therefore, node $x$ must coincide with node
$o$. This contradicts our assumption.\\
Suppose $a$ comes from a diamond. Since the degrees of node
\textbf{1} and $o$ are at least 4, $a$ must be the mid-edge.
Therefore, the diamond must contain $\lambda,\gamma,b,d$. Moreover,
the degrees of node \textbf{2,3} must be two. This is a
contradiction.\\\\
We can also rule out the possibility that $a$ comes from a square
since both its endpoints have degree 4. To conclude, $a$ must come from a triangle. Otherwise, $G$
is undecomposable. \\
Suppose $a,\tau$ come from the same triangle, then the third edge of this block
is annihilated by another edge, denoted as $\delta$. Hence
$\delta,\eta,b,c,d$ come from the same block. This is impossible. If
$a,b$ come from the same triangle, then $c,d$ come from another
triangle, denoted as $\triangle_1$. Thus, $\tau,\gamma$ come from the
same block, which must be a triangle. Denote it as $\triangle_2$.
Notice that its third edge is annihilated since $x$ is connected
only to two of nodes \textbf{1,2,3,4}. But this is impossible
since node \textbf{3} is already black after gluing $\triangle_1$ to
$\triangle_2$. The similar argument shows that $a,c$
can not come from the same triangle. Thus, the graph is
undecomposable when $x$ is connected only to nodes \textbf{1,4}.\\\\
Assume that $x$ is connected to nodes \textbf{1,3} by edges
$\tau,\eta$ respectively. If $a,c$ come from the same triangle, then
$b,d$ come from another triangle, denoted as $\triangle_1$. Thus,
$\tau,\lambda$ come from another block. This block can not be a diamond since $\lambda$ must then
be the mid-edge and the boundary edge $\overline{x2}$ is
annihilated, which is impossible. Hence this block must be a
triangle, denoted as $\triangle_2$. Since $x$ is connected only to two of nodes $1,2,3,4$, the
third edge of $\triangle_2$ is annihilated. However, node \textbf{2}
is already black after gluing $\triangle_1$ to $\triangle_2$. This is
a contradiction. Therefore if $G$ is decomposable, $a,b$ must come
from the same triangle and $c,d$ come from another triangle. Apply
the corresponding replacement as in Figure \ref{twotrisimp}.\\\\
Assume $x$ is connected to at least three of nodes \textbf{1,2,3,4}.
Up to an index relabeling, there are two cases: Either $x$ is connected
to nodes \textbf{1,2,3} or
$x$ is connected to nodes \textbf{1,3,4}.\\\\
Suppose $x$ is connected to nodes \textbf{1,3,4} by $\tau,\rho,\eta$
respectively. Assume that $a,b$ come from the same triangle, then
$c,d$ come from one triangle too. Thus, $\eta,\overline{24}$ come
from the same block, which must be a triangle. Thus, node $x$ and
node \textbf{2} must be connected according to previous argument.
The argument is similar when $a$ and $c$ come from the same triangle. In both cases, the
neighborhood is a disjoint
connected component, see Figure \ref{bigdisjoint}. Otherwise, $G$ is undecomposable.\\\\
Assume $x$ is connected to nodes \textbf{1,2,3} by $\tau,\xi,\rho$
respectively. Assume that $a,b$ come from the same triangle, then
$c,d$ come from another triangle, denoted as $\triangle_1$.
Therefore, $\tau,\gamma,\rho$ must come from the same triangular block, denoted as
$\triangle_2$. Notice that node $x$ is black after gluing
$\triangle_1$ to $\triangle_2$, node $x$ and node \textbf{4} must be
connected by the third edge of $\triangle_2$. Similarly, assuming
that $a,c$ come from the same triangle will also result in the same
neighborhood. In this case, the neighborhood is a disjoint
connected component, see Figure \ref{bigdisjoint}. Otherwise, the
graph is undecomposable.
\begin{figure}
\centering
 \begin{minipage}[b]{0.24\linewidth}
 \centering
 %auto-ignore
% Generated with LaTeXDraw 2.0.8
% Fri Jul 09 15:21:53 PDT 2010
% \usepackage[usenames,dvipsnames]{pstricks}
% \usepackage{epsfig}
% \usepackage{pst-grad} % For gradients
% \usepackage{pst-plot} % For axes
\scalebox{1.2} % Change this value to rescale the drawing.
{
\begin{pspicture}(1,-1.0179597)(3.4477038,2.4073527)
\psline[linewidth=0.04cm,arrowsize=0.05291667cm 2.0,arrowlength=1.4,arrowinset=0.4]{<-}(1.5086414,0.7064153)(2.2886415,-0.0735847)
\psline[linewidth=0.04cm,arrowsize=0.05291667cm 2.0,arrowlength=1.4,arrowinset=0.4]{<-}(2.2886415,-0.0735847)(3.0886414,0.6864153)
\psline[linewidth=0.04cm,arrowsize=0.05291667cm 2.0,arrowlength=1.4,arrowinset=0.4]{<-}(2.2886415,-0.0735847)(3.0886414,-0.8735847)
\psline[linewidth=0.04cm,arrowsize=0.05291667cm 2.0,arrowlength=1.4,arrowinset=0.4]{<-}(1.4886414,-0.8935847)(2.2886415,-0.0935847)
\psline[linewidth=0.04cm,arrowsize=0.05291667cm 2.0,arrowlength=1.4,arrowinset=0.4]{<-}(3.0686414,0.7064153)(1.5086414,0.7064153)
\psline[linewidth=0.04cm,arrowsize=0.05291667cm 2.0,arrowlength=1.4,arrowinset=0.4]{<-}(3.0686414,-0.8935847)(1.5286413,-0.8935847)
\psdots[dotsize=0.12](2.2886415,-0.0935847)
\psdots[dotsize=0.12](3.0886414,0.7064153)
\psdots[dotsize=0.12](1.4686414,0.7064153)
\psdots[dotsize=0.12](1.4686414,-0.9135847)
\psdots[dotsize=0.12](3.0886414,-0.9135847)
\rput{-84.55601}(2.2112436,0.9161459){\psarc[linewidth=0.04,arrowsize=0.05291667cm 2.0,arrowlength=1.4,arrowinset=0.4]{<-}(1.6094258,-0.757927){1.4760472}{80.36318}{180.0}}
\rput{1.4405128}(-0.019636098,-0.07390145){\psarc[linewidth=0.04,arrowsize=0.05291667cm 2.0,arrowlength=1.4,arrowinset=0.4]{<-}(2.9294257,-0.817927){1.4760472}{80.36318}{180.0}}
\usefont{T1}{ptm}{m}{n}
\rput(2.289579,0.9164153){$\lambda$}
\usefont{T1}{ptm}{m}{n}
\rput(3.24286,-0.3235847){$\gamma$}
\psline[linewidth=0.04cm,arrowsize=0.05291667cm 2.0,arrowlength=1.4,arrowinset=0.4]{<-}(1.4286413,0.7064153)(2.2886415,2.1264153)
\psline[linewidth=0.04cm,arrowsize=0.05291667cm 2.0,arrowlength=1.4,arrowinset=0.4]{<-}(2.2686412,2.1464152)(3.1086414,0.7264153)
\psline[linewidth=0.04cm,arrowsize=0.05291667cm 2.0,arrowlength=1.4,arrowinset=0.4]{<-}(1.4486413,-0.8535847)(2.2686412,2.1064153)
\psline[linewidth=0.04cm,arrowsize=0.05291667cm 2.0,arrowlength=1.4,arrowinset=0.4]{<-}(2.2686412,2.1064153)(3.0686414,-0.8735847)
\usefont{T1}{ptm}{m}{n}
\rput(1.2955164,0.79641527){1}
\usefont{T1}{ptm}{m}{n}
\rput(3.267235,0.7764153){2}
\usefont{T1}{ptm}{m}{n}
\rput(3.2962976,-0.8635847){3}
\usefont{T1}{ptm}{m}{n}
\rput(1.3095789,-0.8435847){4}
\usefont{T1}{ptm}{m}{n}
\rput(2.27411,2.3164153){x}
\usefont{T1}{ptm}{m}{n}
\rput(1.7172351,1.5564153){$\tau$}
\usefont{T1}{ptm}{m}{n}
\rput(1.4914539,0.2564153){$\eta$}
\usefont{T1}{ptm}{m}{n}
\rput(2.8519225,1.5764153){$\xi$}
\usefont{T1}{ptm}{m}{n}
\rput(3.0612977,0.2964153){$\rho$}
\end{pspicture}
}
 \caption{}\label{bigdisjoint}
 \end{minipage}
\end{figure}
\\

%Notice that in this case, we require only that node \textbf{2} is
%connected towards node \textbf{3} or towards node \textbf{4}, i.e.
%this disjoint component is undecomposable only if edge
%$\overline{23}$, $\overline{24}$ are both directed towards node
%\textbf{2}.
\noindent Suppose that node \textbf{4} is connected only to one of
nodes \textbf{2,3}.\\
Assume that node \textbf{4} is connected to node \textbf{3}. In this case, if edge $\overline{34}$
is directed towards node \textbf{4}, then edges $c,d,\overline{34}$
can not form a triangular block. This means edges $b,d$, edges $a,c$ must come from two triangle, otherwise, the graph is undecomposable. Since by assumption, nodes \textbf{2,4} are
not connected, the corresponding edge is annihilated
by another edge, denoted as $\eta$. Hence, $\eta,\lambda$ come from
the same block which must be a triangle. So node \textbf{4} is
black before attaching edge $\overline{34}$. This is a contradiction.
In this
case, the graph is undecomposable.\\
If edge $\overline{34}$ is directed towards node \textbf{3}, there
are two possibilities: edges $c,d,\overline{34}$ form a triangular
block or edges $b,d$ come from a triangle $\triangle$. \\
Suppose it's the latter case, the edge of $\triangle$ that connects nodes
\textbf{2,4} is annihilated by another edge $\tau$. This forces edges
$\tau,\overline{34},\lambda$ to form a block. This is impossible. So
edges $c,d,\overline{34}$ form a
triangle. Otherwise, the graph is undecomposable. Hence, edges $c,d\overline{34}$ form a triangular block. We apply the similar
replacement as in Figure \ref{twotrisimp}. \\\\
Assume that nodes \textbf{3,4} and nodes \textbf{2,4} are
disconnected. Then one edge of the triangular block that contains edge
$b$ is annihilated by another edge, denoted as $\tau$. If $a,b$ come from the same triangle, then $\tau$
connects nodes \textbf{3,4}. Therefore, $\gamma,\tau$ must form a triangle. This means that nodes
\textbf{1,4} must be connected and nodes \textbf{2,3} are
disconnected. Similarly, if $\lambda,\tau$ must form a triangle, then $\tau$ connects nodes \textbf{2,4}, Thus, nodes \textbf{2,3} must
be disconnected and nodes \textbf{1,4} are connected.
(Figure \ref{CaseAB}). Notice that in Case A, node \textbf{2} may have degree larger than
two. And in Case B, node \textbf{3} may have degree larger than 2. Thus, it suffices to examine the degrees of
nodes \textbf{2,3} to determine whether edges $a,b$ or edges $a,c$ come from a
triangular block, then apply the corresponding replacement. To be more
precise, if degree of node \textbf{2} is at least 3, it is Case
A; if degree of node \textbf{3} is at least 3, it is Case B; if
both have degree 2, either decomposition is possible, and the neighborhood is a disjoint connected
component.

\begin{figure}[tb]
\centering
 \begin{minipage}[c]{0.5\linewidth}
 \centering
 \includegraphics[width=0.3\linewidth]{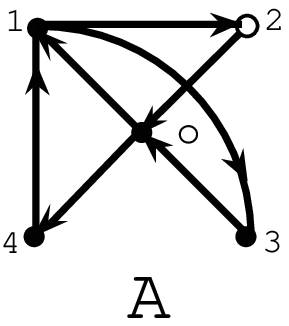}
 \includegraphics[width=0.3\linewidth]{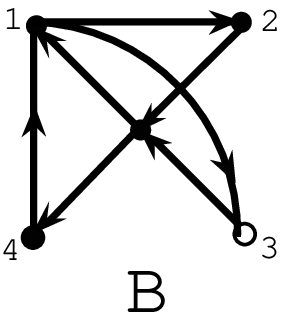}
 \caption{}\label{CaseAB}
 \end{minipage}
\end{figure}
\subsection{Node \textbf{1} is Connected to Node \textbf{2}, Disconnected from Node \textbf{3}}
Assume that nodes \textbf{1,2} are connected by edge $\lambda$ but
nodes \textbf{1,3} are not connected. If $\lambda$ is directed from
node \textbf{2} to node \textbf{1}, then $a,b,\lambda$ do not form a
triangular block. Moreover, $a$ does not come from a diamond. If $a$ comes from a spike,
$b,c,d$ must come from a diamond and node \textbf{4} is connected only to nodes \textbf{2,3}. Since by assumption,
nodes \textbf{1,2} are connected, node \textbf{1} must coincide with node \textbf{4}. This means nodes \textbf{1,3} are connected, and it
contradicts our assumption. Therefore,
$a$ must come from a triangle that does not contain $b,\lambda$. Hence the block must contain $c$. The third edge in that triangle is
annihilated by another edge, denoted
as $\tau$. Moreover, $\tau$ is directed from node \textbf{3} to node
\textbf{1}. Thus, $\tau,\lambda$ come from the same block. However,
there is no such block with such directions. Hence in
this case, the graph is undecomposable. \\\\
Assume that $\lambda$ is directed from node \textbf{1} to node
\textbf{2}.\\\\
If $a$ comes from a spike, $b,c,d$ come from the same block, which
must be a diamond. Therefore, nodes \textbf{1,2} must be disconnected. This means node \textbf{1} coincides with nodes \textbf{4}. Therefore nodes \textbf{1,3} are connected. Contradiction.\\
If $a$ comes from a diamond, there are two cases.
\begin{itemize}
\item \textbf{1,3} and nodes \textbf{3,4} are disconnected.
The diamond contains $b,c,\lambda$. Nodes \textbf{1,3} must be connected. Contradiction.
\item The diamond contains $b,d,\lambda$. Notice
that in this case, nodes \textbf{2,4} are connected, node
\end{itemize}
\begin{lma}
Assume that nodes \textbf{1,2} are connected by edge $\lambda$
directed from node \textbf{1} to node \textbf{2}, nodes \textbf{1,3} are disconnected, nodes \textbf{2,4}
are connected by an edge directed from node \textbf{4} to
\textbf{2}, the degrees of nodes \textbf{1,4} are two.
\begin{enumerate}
\item If $G$ is decomposable and nodes \textbf{2,3} is disconnected, then $a$ comes from a
diamond containing $a,b,d,\lambda$ and $c$ comes from a spike.
\item Assume nodes \textbf{2,3} are connected by an edge directed
from node \textbf{2} to \textbf{3}: \begin{enumerate}
\item If the degree of node \textbf{3} is two, then the neighborhood is a disjoint
connected component.
\item If the degree of node \textbf{3} is at least three, then the graph is not decomposable.
\end{enumerate}
\end{enumerate}
\end{lma}
\begin{proof} \textbf{1:} Suppose nodes \textbf{2,3} are disconnected and the
statement is false.\\
It is easy to rule out the possibility that $a$ comes from a square
or a fork. \\
If $a$ comes from a spike, $b,c,d$ comes from a diamond
and node \textbf{2} is black. Thus, $\lambda$ can not be attached unless node \textbf{1} coincide with node \textbf{4}.
Thus degree of node \textbf{1} coincide with node \textbf{4} is 4. This contradicts the assumption that degree of node \textbf{1} is 2.
Suppose edge $a$ comes from a diamond. Since the statement
is false, the diamond must contain $a,b,c$. Hence node \textbf{1} must be
connected to
node \textbf{3}. This ia a contradiction to our assumption.\\
If $a,b$ come from a triangle, $c,d$ also come from a triangle denoted by
$\triangle$. Thus, the third edge of $\triangle$ is annihilated by
another edge, denoted as $\eta$. Hence $\eta,\overline{24}$ come
from the same block, which must be a triangle. Thus, node
\textbf{2,3} must be connected. Contradiction.\\
If $a,c$ come from a triangle, then the third edge of this triangle
is annihilated by another edge, again denoted as $\eta$. Hence
$\eta,\lambda$ must come from the same block, which must be a
triangle. Hence nodes \textbf{2,3} must be connected.
Contradiction.\\\\
\textbf{2:} Assume nodes \textbf{2,3} are connected by an edge directed
from node \textbf{2} to \textbf{3}. It suffices to check the cases when $a$ comes from a
triangle or a diamond. In the first case, by previous
argument, all nodes in this neighborhood are already black. This proves $(a)$. In the second case, the degree of node \textbf{3} is three. The
only possibility to obtain a decomposition is to glue a triangle or diamond on nodes \textbf{2,3}.
In either case, the degree of node \textbf{2} is larger than 4. This contradicts the assumption of the section
that the degree of any node of $G$ is at most 4. Hence the graph is undecomposable. This proves $(b)$
\end{proof}
\noindent Suppose node \textbf{1} or node
\textbf{4} has degree at least three, then $a$ doesn't come
from a diamond. Moreover, we can rule out the possibility that $a$
comes from a square, since this will force node \textbf{1} to be connected
to node \textbf{3} with an edge \textbf{1}$\rightarrow$\textbf{3}.
Hence $a$ must come from a triangle. There are two cases. In case A,
edges $a,b$ are in the same triangle, thus, edges $c,d$ are in the
same triangle. In this case, apply the same replacement as in Figure
\ref{twotrisimp}. In case B, edges $a,c$ are in the same
triangle denoted as $\triangle$. Thus, edges $b,d$ are in the same
triangle. The edge in $\triangle$ that connects nodes
\textbf{1,3} is annihilated by an edge denoted as $\lambda$. Thus,
$\gamma,\lambda$ come from the same block, and it must be a
triangle. This means nodes \textbf{2,3} are connected. Moreover, in
Case B, nodes \textbf{2,4} must be connected by an edge directed from node \textbf{4} to \textbf{2}
and nodes \textbf{1,2,3} are black. Notice that nothing has been glued to node
\textbf{4} yet, we use this to identify case B.
\begin{lma}
Suppose that graph G is decomposable and node \textbf{1} is connected to node \textbf{2} by an oriented edge \textbf{1}$\rightarrow$\textbf{2}, nodes \textbf{1,3} are disconnected:
\begin{enumerate}
\item[a] Suppose nodes \textbf{2,4} and nodes \textbf{2,3} are connected. If degree of node \textbf{4}
is at least 3, then $a,c$ come from one triangle and $b,d$
come from another triangle. If degree of node \textbf{1} is at least 3, then $a,b$ come from the same triangle and $c,d$ come from
the same triangle. (Figure \ref{tridistinguishlma}). If the degrees of both nodes
\textbf{1} and \textbf{4} are 2, then the neighborhood is a disjoint connected
component.
\item[b] If nodes \textbf{2,4} or \textbf{2,3} are
disconnected, then $a,b$ come from one triangle, $c,d$ come from another triangle.
\end{enumerate}
\end{lma}
\begin{figure}[bt]
\centering
\includegraphics[width=0.2\linewidth]{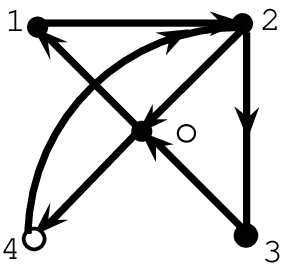}
\includegraphics[width=0.2\linewidth]{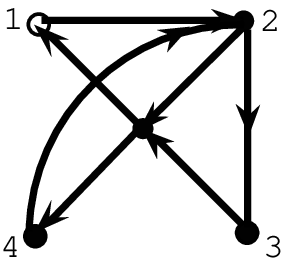}
\caption{Case B}\label{tridistinguishlma}
\end{figure}
\begin{proof}
According to the previous argument, it suffices to prove part $a$.
Suppose nodes \textbf{2,4} and nodes \textbf{2,3} are connected. If
$a,b$ come from one triangle, then $c,d$ come from another
triangle. Moreover, edges $\overline{24},\overline{23}$ must come
from the same block, which must be a triangle. The third edge of
this triangle annihilates $\overline{34}$. Therefore, node
\textbf{4} have degree 2. Thus, if the degree of node \textbf{4} is
at least 3, $a,c$ must come from one triangle. Otherwise,
the graph is undecomposable. The rest of part $a$ follows from the
previous argument.
\end{proof}
\subsection{Node \textbf{1} is Disconnected from Nodes \textbf{2,3}}
Assume that neither nodes \textbf{1,2} nor nodes \textbf{1,3} are connected. Without loss of generality, we can assume that neither nodes
\textbf{3,4} nor nodes \textbf{2,4} are connected. Otherwise, we
can relabel the indices of boundary nodes and apply the previous
argument.\\\\
Assume that $a,b$ come from the same triangle. Then the third edge
of it is annihilated by another edge, denoted as $\lambda$.
This edge $\lambda$ can be a part of a spike, a triangle, or the mid-edge of a
diamond block. If it comes from a triangle or a diamond, nodes $1,2$
must both be connected to
another node $x$.\\\\
Conversely, it is possible to determine whether $a,b$ or $a,c$ come
from the same triangle by considering nodes connecting some of nodes
\textbf{1,2,3,4}. \\\\
Suppose none of nodes \textbf{1,2,3,4} is connected to any other node
except for $o$, then the neighborhood is a disjoint connected
component.\\\\
Assume that nodes \textbf{1,4} or \textbf{2,3} are both connected to
the same node $x$. We assume that vertex $x$ is distinct from nodes
 \textbf{1,2,3,4}, and $o$. Without loss of generality, assume
nodes \textbf{1,4} are connected to $x$. Denote $\alpha=\overline{1x}$ and
$\beta=\overline{4x}$. If $a,d$ come from the same block, it must be
a diamond. Therefore, $b,c$ come from another diamond. Since degree
of $o$ is 4, the mid-edges of these two diamonds annihilate each other. Then
the neighborhood is a disjoint connected component, see Figure
\ref{doublediamond2}. If we rule out this case, $a,d$
must come from two blocks. By the previous argument, neither $a$ nor
$d$ comes from the a diamond. Thus, they must come from triangular blocks. Moreover, the triangle that contains $a$ must contain $c$ or
$b$. Without loss of generality, assume $c$ is contained in such triangle $\triangle$. Then the third side of
$\triangle$ is annihilated by another edge, denoted as $\tau$.
Similarly, $b,d$ come from another triangle $\triangle_1$. The third edge of $\triangle_1$ is
annihilated by an edge denoted as $\eta$. Then $\alpha,\tau$ come
from the same block, which must be a triangle. If node \textbf{3} and
$x$ are not connected, the third edge of the triangle
containing $\alpha,\tau$ must be annihilated. This is impossible since
node \textbf{3} is already black. Therefore, in this case, the graph is undecomposable. If nodes \textbf{3} and $x$ are connected by an edge
denoted as $\gamma$, then $\alpha,\tau,\gamma$ form a triangle if
their directions match a triangle, otherwise, the graph is undecomposable. Notice that $\beta,\eta$ must come from the same
block. Thus, it must be a triangle if their directions match, otherwise, the graph is undecomposable. In this case,
nodes $x$,\textbf{2} must be connected, and the neighborhood is a
disjoint connected component, see Figure \ref{twist}. In
this case, there is an alternative decomposition , see Figure
\ref{twistdecom}.\\
\begin{figure}[bt]
\centering
 \begin{minipage}[b]{0.3\linewidth}
 \centering
 \includegraphics[width=0.8\linewidth]{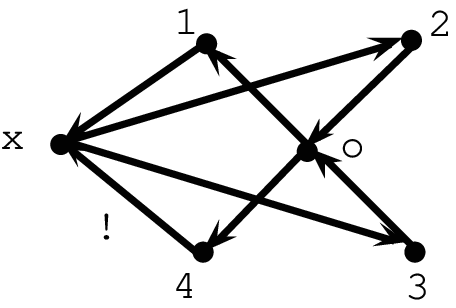}
 \caption{}\label{twist}
 \end{minipage}
 \begin{minipage}[b]{0.5\linewidth}
 \centering
 \includegraphics[width=0.6\linewidth]{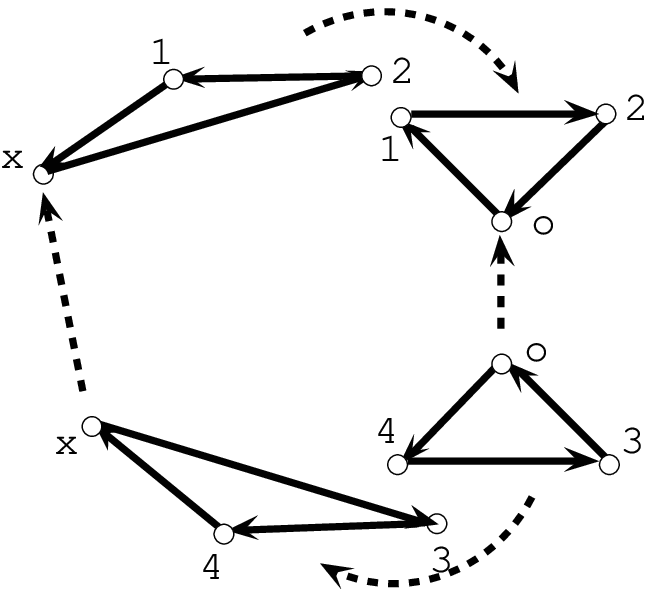}
 \caption{}\label{twistdecom}
 \end{minipage}\\
\end{figure}

\noindent Suppose $x$ is connected to nodes \{1,2\} (resp. \{1,3\},\{3,4\},\{2,4\}),
we claim that $a,b$ (resp. $a,c$, $c,d$, $b,d$) come from one triangle, therefore $c,d$ (resp. $b,d$,
$a,b$, $a,c$) must come from another
triangle.\\\\
Assume that nodes \textbf{1,2} are connected to node $x$ and
suppose $a,b$ do not come from the same triangle in any of the
possible decomposition of $G$. Denote $\alpha=\overline{1x}$,
$\beta=\overline{2x}$. Notice that from the previous argument, $a$
does not come from a spike, a fork, a diamond or a square. So it
must be contained in a triangular block $\triangle$. If $\triangle$
contains $\alpha$, then the third edge is annihilated by another
edge, denoted as $\tau$. Hence, $\tau,b,c,d$ come from the same
block which must be a square, also nodes $o,x$ must be colored white in that
block. This is impossible. Thus, $\triangle$ does not contain
$\alpha$. So it must contain $c$. Then the third edge must also be
annihilated by another edge, again denoted as $\tau$. In this case,
$\tau,\alpha$ must come from the same block, which must be a triangle
$\triangle_1$. If nodes \textbf{3} and $x$ are connected, we will get a neighborhood similar as the one in Figure \ref{twist}. As we already know, there is an alternative decomposition in which $a,b$ come from the same
triangle. If nodes \textbf{3} and $x$ are not
connected, then the third edge of $\triangle_1$ is annihilated by
another edge. However, node \textbf{3} is already black after gluing
$\triangle_1$.  This means the third edge of $\triangle_1$ can not
be annihilated. This is a contradiction. Therefore, $a,b$ come from a
same triangle in a decomposition of $G$. Otherwise, $G$ is not
decomposable.
\begin{rmk}
If $x$ is connected to nodes \textbf{1,2} and the graph is decomposable, except for one case when the neighborhood is a disjoint connected component,
there is a unique decomposition in which $a,b$ comes from a triangular block and $c,d$ comes from another.
\end{rmk}
\subsection{Summary}
All possible neighborhoods of nodes with degree 4 are listed in
Figure \ref{tree}.
\begin{figure}
\centering
\includegraphics[width=1\linewidth]{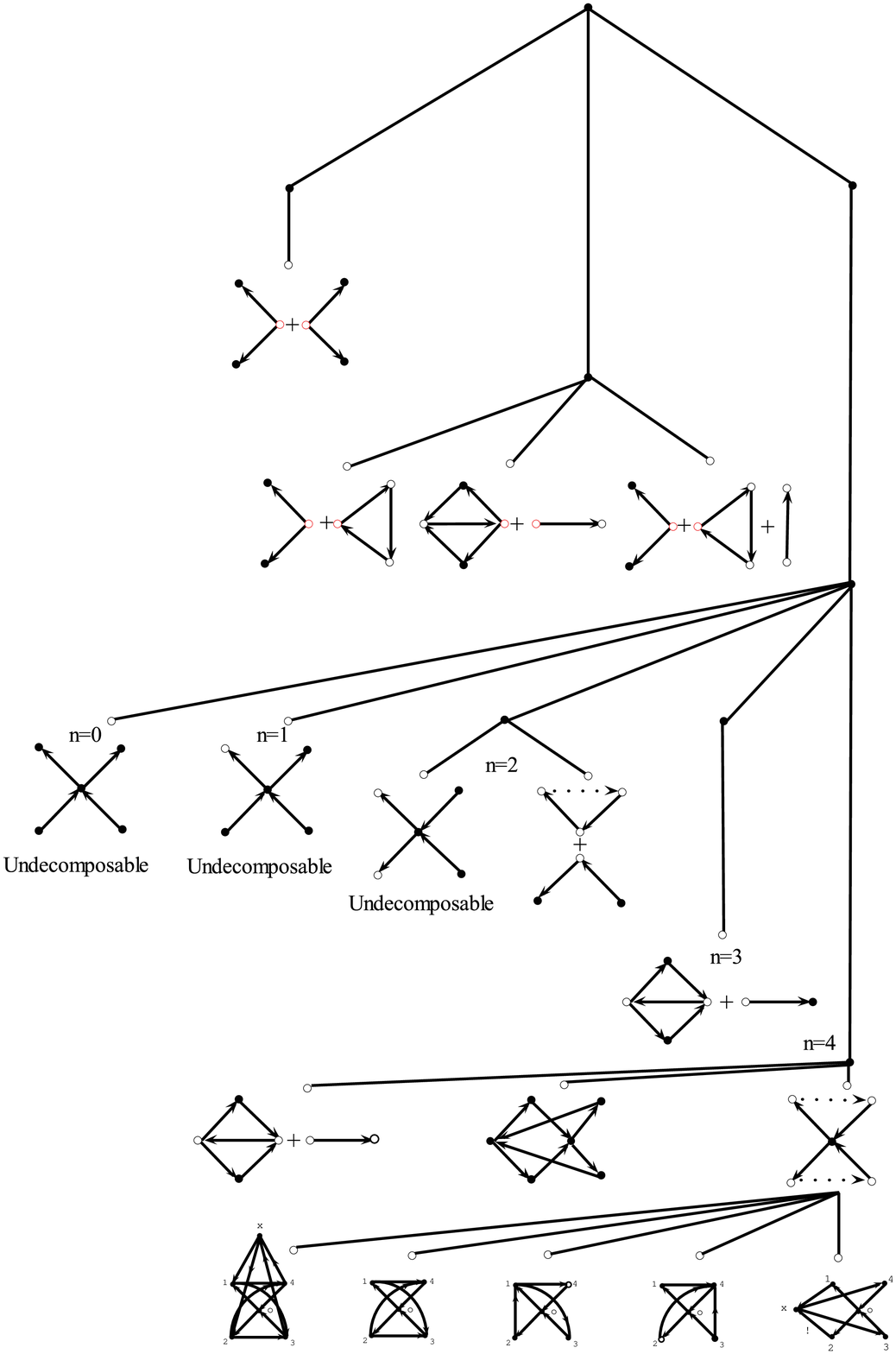}
\caption{Neighborhoods of Nodes with Degree 4}\label{tree}
\end{figure}
%=================From lsod======================
\section{Simplification on Nodes of Degree 3}\label{deg3node}
Assume the neighborhoods of nodes of degree at least 4 are all
simplified, and every node in the graph has degree at most 3.
We focus on the nodes of degree $3$.\\\\
\begin{figure}[hbtc]
 \centering
 \begin{minipage}[c]{0.2\linewidth}
 \centering
 \includegraphics[width=0.9\linewidth]{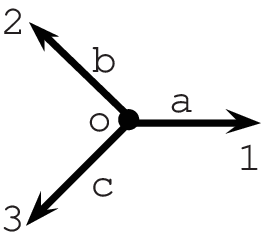}
 \caption{}\label{3out}
 \end{minipage}
 \begin{minipage}[c]{0.5\linewidth}
 \centering
 \includegraphics[width=0.3\linewidth]{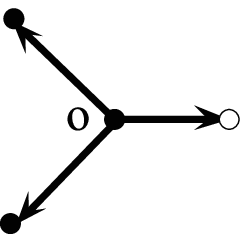}\raisebox{3em}{ =}
 \raisebox{-0.5em}{\includegraphics[width=0.2\linewidth]{outfork.eps}}\raisebox{3em}{ +}
 \raisebox{4em}{\includegraphics[angle=-90,width=0.3\linewidth]{spike.eps}}
 \caption{}\label{forkspike}
 \end{minipage}
 \begin{minipage}[c]{0.2\linewidth}
 \centering
 \includegraphics[width=0.8\linewidth]{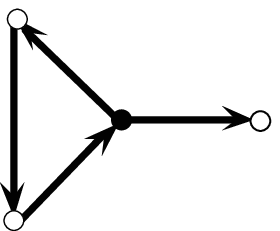}
 \caption{}\label{forkspikesub}
 \end{minipage}
\end{figure}

\subsection{All edges have the same direction.}
Without loss of generality, assume that they are all directed outwards.
(Figure \ref{3out}) Suppose one of them (denoted by $a$) comes from a triangle.
Since deg($o$)=3, and neither of the rest two edges comes from the same
triangle, the incoming edge incident to $o$ in the
same triangle must be annihilated. Denote this edge as $e$. Then $e$ must be annihilated by
an outward edge from another block. This block must contain the
remaining outward edges $b$ and $c$. But there is no such block.
This is a contradiction. Hence edge $a$ is not contained in a triangular block.
By symmetry, none of the three edges comes
from a triangle\\\\
If one of them comes from a fork, one of the remaining two edges
must also belong to the same fork. Thus, the third edge must come
from a single spike. (Figure \ref{forkspike}) Otherwise, the graph is undecomposable.
Replace the neighborhood with Figure \ref{forkspikesub}. By Lemma \ref{trianglespike}, this
replacement is reversible.\\
\begin{rmk}
In order to apply the replacement, we need to identify which two edges come from a fork. This can be done by checking the degrees of boundary nodes. If one of the boundary nodes has degree more than $1$, the corresponding edge must come from a spike and the remaining two edges form a fork. If all boundary nodes have degree $1$, we have a disjoint connected component and the decomposition is non-unique. If at least two of the boundary nodes have degrees more than $1$, the graph is undecomposable.
\end{rmk}
If one of the edges, denoted by $a$, comes from a diamond, denoted as $\diamondsuit$, then one of
remaining, denoted by $b$, must come from the same diamond. Thus the third edge $c$ is not contained in the same block. Since the degree of
$o$ is 3, the mid-edge in $\diamondsuit$ must be annihilated by
another edge directed away from $o$, denoted as $e$. Thus, $c,e$
come from the same block. Since the degree of $o$ is 3, the block
containing $c,e$ must be a triangle. However the directions of these
two edges do not match a triangle. This is a contradiction.\\\\
To conclude, if all three edges incident to a node are all directed inwards or outwards and the graph is decomposable, the neighborhood must be obtained from
gluing a fork and a spike.
\subsection{Two outward edges $+$ One inward edge}
See Figure \ref{2o1in}. Assume edge $a$ is directed inwards with endpoints nodes \textbf{1} and $o$. If $a$ comes from a spike, the
remaining two edges must come from a fork. If $a$ comes from a diamond $\diamondsuit$, the block
must contain at least $b$ or $c$ since only the mid-edge can be annihilated in a diamond. Assume $b$ is contained in this block.
The directions of $a,b$ force $c$ to be contained in $\diamondsuit$,
as shown in Figure \ref{sd}. Since all nodes in $G$ has degree at most 3, this diamond must be a disjoint connected
component. Otherwise, $G$ is undecomposable.\\\\
Assume $a$ comes from a triangle $\triangle$, there are two
cases:\\
\begin{figure}[tcb]
 \centering
 \begin{minipage}[b]{0.3\linewidth}
 \centering
 \includegraphics[width=0.6\linewidth]{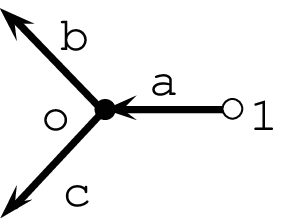}
 \caption{}\label{2o1in}
 \end{minipage}
 \begin{minipage}[b]{0.3\linewidth}
 \centering
 \includegraphics[width=0.6\linewidth]{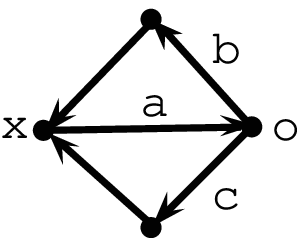}
 \caption{}\label{sd}
 \end{minipage}
 \begin{minipage}[b]{0.3\linewidth}
 \centering
 \includegraphics[width=0.6\linewidth]{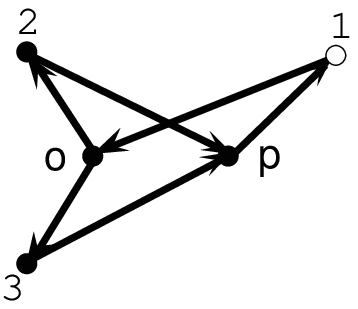}
 \caption{}\label{diamondtriangle}
 \end{minipage}
\end{figure}
\noindent\textbf{Case 1:} $\triangle$ contains neither of $b,c$. Then $b,c$ must
come from a diamond or a fork. In the latter, the
remaining edge of $\triangle$ that is incident to $o$ must be
annihilated. This forces $o$ to be a black node even before $b,c$
are attached. This is a contradiction. Therefore $b,c$ come from a diamond
$\diamondsuit$. Notice that the mid-edge in $\diamondsuit$ should be
annihilated by an edge of $\triangle$, as shown in Figure
\ref{diamondtriangle}. Simplify the neighborhood by removing the diamond block and leaving the triangle $\triangle$ containing $a$. (See Figure \ref{dtdecom}.)\\
\textbf{Case 2:} $\triangle$ contains one of $b,c$. Without loss of
generality, assume it is $b$. Then $c$ must come from a spike.\\
\begin{figure}[btch]
\centering
\begin{minipage}{0.6\linewidth}
\centering
\includegraphics[width=0.7\linewidth]{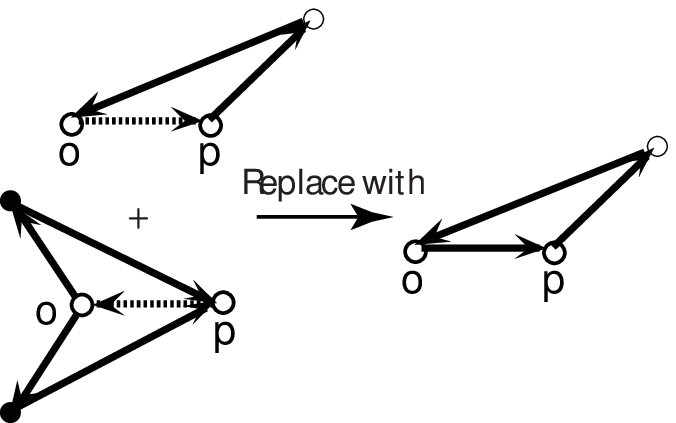}
\caption{}\label{dtdecom}
\end{minipage}
\end{figure}
Let's take deeper look at Case 2. Assume that the third
edge of $\triangle$
is $d$. There are two possibilities.\\
\vspace{-1em}
\begin{enumerate}
\item[\textbf{a.}] Edge $d$ is annihilated in the graph.
\item[\textbf{b.}] Edge $d$ is not annihilated in the graph.(Figure \ref{case12})
\end{enumerate}
Next, start with case \textbf{a}. There are three ways to annihilate $d$.\\\\
\noindent\textbf{Case a1} Edge $d$ is annihilated by a single
spike. See Figure \ref{case111}. Then replace this neighborhood by
the one in Figure \ref{case111sub}. By Lemma \ref{trianglespike}, this replacement is reversible.\\
\begin{figure}[btch]
 \centering
 \begin{minipage}[b]{0.3\linewidth}
 \centering
 \includegraphics[width=0.6\linewidth]{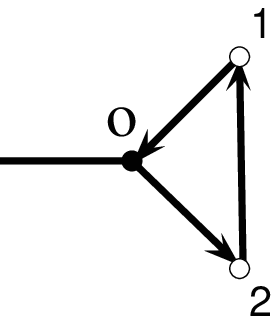}
 \caption{}\label{case12}
 \end{minipage}
 \begin{minipage}[b]{0.3\linewidth}
 \centering
 \includegraphics[width=0.6\linewidth]{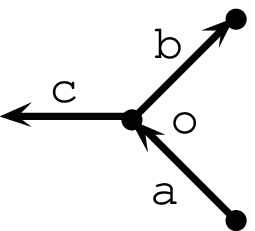}
 \caption{}\label{case111}
 \end{minipage}
 \begin{minipage}[b]{0.3\linewidth}
 \centering
 \includegraphics[width=0.6\linewidth]{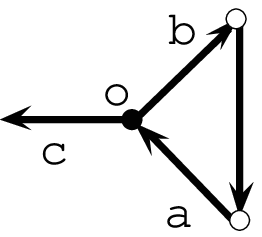}
 \caption{}\label{case111sub}
 \end{minipage}
\end{figure}

\noindent\textbf{Case a2} Edge $d$ is annihilated by
one edge of a triangle. See Figure \ref{case112}. If $p$ is connected to
$o$ via edge $c$, then this graph forms a disjoint
connected component. (Figure \ref{case112deg3blackatp}) Otherwise,
$G$ is undecomposable. If $c$ does not connect $p$, and there is
nothing else connected to $p$ (deg($p$)=2). Then we replace the neighborhood with the one in Figure \ref{case111sub}
%Disconnect edge $b$ at
%node $o$, then Figure \ref{case112deg3blackatp} is replaced by (or
%more precisely, expands to) Figure \ref{case112expand}.
If $c$ does
not connect $p$, and deg($p$)=3, as shown in Figure
\ref{case112deg3atp}. There are two cases:\\
\begin{itemize}
\item In Figure \ref{case112deg3atp}, suppose node \textbf{3} coincides with node $x$,
edge $\overline{px}$ is directed from $x$ to $p$ and deg($1$)=2, the neighborhood in Figure \ref{case112deg3atp} coincides with Figure \ref{diamondtriangle}. In this case, if graph $G$ is decomposable, the neighborhood is a disjoint connected component.
\item Suppose node \textbf{3} is not coincide with node $x$,  Then the neighborhood can replaced by Figure \ref{doublediamondsimp}. It's similar if edge $\overline{px}$ is directed from $p$ to $x$. This replacement it reversible by previous lemma.\\
\end{itemize}
\begin{figure}[btch]
 \centering
 \begin{minipage}[c]{0.3\linewidth}
 \centering
 \includegraphics[width=0.8\linewidth]{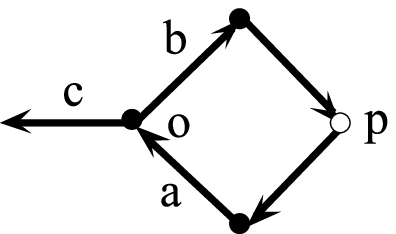}
 \caption{}\label{case112}
 \end{minipage}
 \begin{minipage}[c]{0.3\linewidth}
 \centering
 \includegraphics[width=0.7\linewidth]{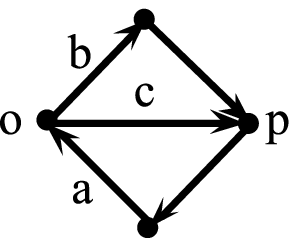}
 \caption{}\label{case112deg3blackatp}
 \end{minipage}
 \\
 \vspace{0.5em}
 \begin{minipage}[c]{0.4\linewidth}
 \centering
 \includegraphics[width=0.8\linewidth]{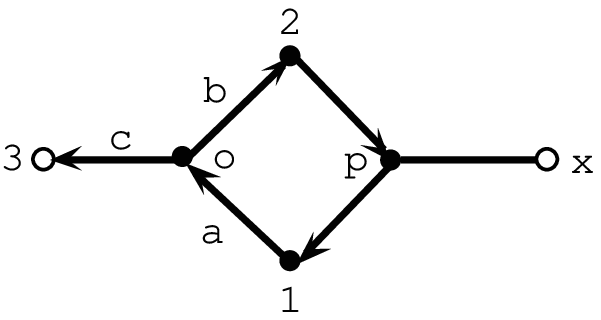}
 \caption{}\label{case112deg3atp}
 \end{minipage}
 \begin{minipage}[c]{0.4\linewidth}
 \centering
 \includegraphics[width=0.8\linewidth]{doublediamondsimp.eps}
 \vspace{1.5em}
 \caption{}\label{doublediamondsimp}
 \end{minipage}
\end{figure}

\begin{figure}[btch]
 \centering
\begin{minipage}[b]{0.4\linewidth}
\centering
 \includegraphics[width=0.6\linewidth]{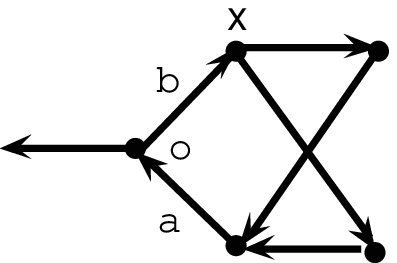}
 \caption{}\label{trianglediamond}
 \end{minipage}
\end{figure}
\noindent \textbf{Case a3} Assume $d$ is annihilated by the mid-edge of a
diamond, see Figure \ref{trianglediamond}.
Replace the neighborhood by the one in Figure \ref{case111sub} as well.\\\\
Next, let's discuss case \textbf{b}. If $d$ is not annihilated, there are three subcases:
\begin{itemize}
\item[\textbf{b1}] Both nodes \textbf{1,2} have degree two. In this case, the neighborhood must come from a spike and a triangle by Lemma~ ~\ref{trianglespike}.
\item[\textbf{b2}] One of the nodes \textbf{1,2} has degree two and the other one has degree three. Assume the degree of node \textbf{2} is two,
and the degree of node \textbf{1} is three. In this case, the neighborhood is shown in Figure \ref{case12b}.
\begin{figure}[btch]
\centering
\begin{minipage}[b]{0.5\linewidth}
\centering
\includegraphics[width=0.7\linewidth]{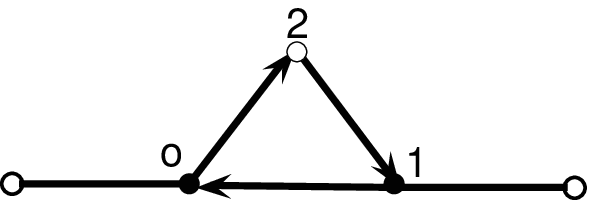}
\end{minipage}
\begin{minipage}[b]{0.3\linewidth}
\centering
\includegraphics[width=0.45\linewidth]{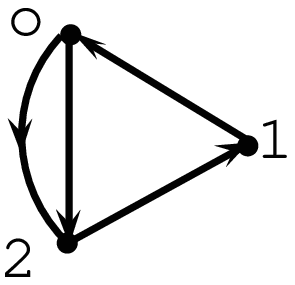}
\end{minipage}
\caption{}\label{case12b}
\end{figure}
\item[\textbf{b3}] Both nodes \textbf{1,2} have degrees three. In this case, we count the number of nodes that are connected to nodes $o$ and \textbf{1,2}, denoted as $n$.
    \begin{itemize}
    \item $n=3$, the only possible decomposable situation is Figure \ref{n3}.
    \item Suppose $n=2$. One of the exterior nodes is connected to two of nodes $o$,\textbf{1,2}. Denote this node by $x$. If $x$ is connected to nodes \textbf{1,2} (resp. $o$,\textbf{1} or $o$, \textbf{2}), then the other exterior node is connected to nodes \textbf{o} (resp. node \textbf{2} or node \textbf{1}). In this case, edges $\overline{x1}, \overline{x2}$ (resp. $\overline{xo}, \overline{x1}$ or $\overline{xo}, \overline{x2}$ ) come from two spikes and degree of $x$ must be two. (See Figure \ref{n2})
\begin{figure}[btch]
\centering
\begin{minipage}[b]{0.7\linewidth}
\centering
\includegraphics[width=0.3\linewidth]{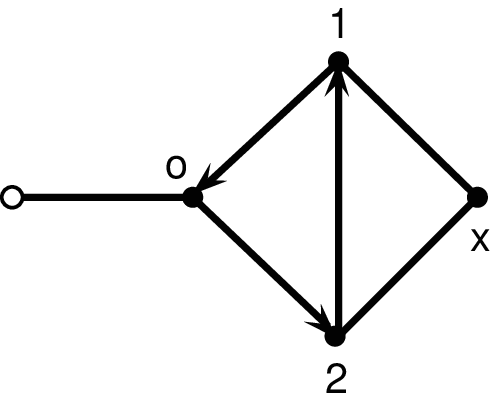}
\includegraphics[width=0.3\linewidth]{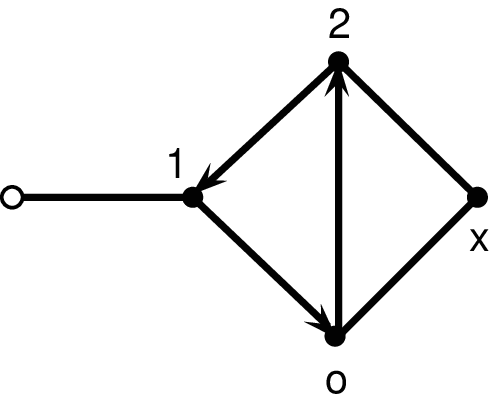}
\includegraphics[width=0.3\linewidth]{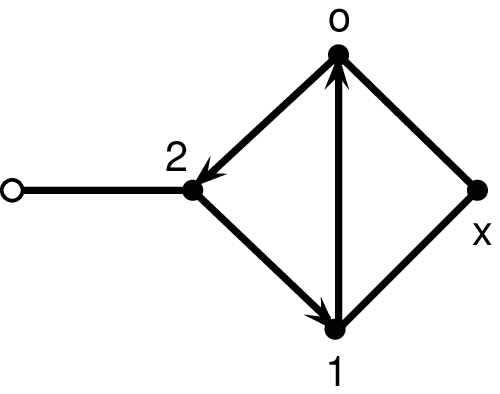}
\caption{}\label{n2}
\end{minipage}
\end{figure}
    \item $n=1$. Notice that we assume that the degree of nodes $o$,\textbf{1,2} are all three. So there are two cases, as shown in Figure \ref{1ext}. Note that Figure \ref{1ext}A is undecomposable, so the only decomposable neighborhood is Figure \ref{1ext}B.
\begin{figure}[btch]
\centering
\begin{minipage}[b]{0.4\linewidth}
\centering
\includegraphics[width=0.4\linewidth]{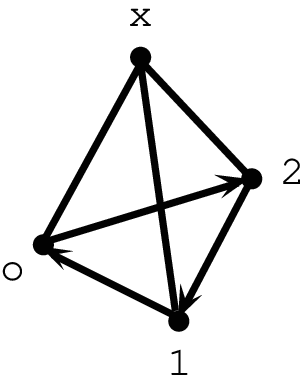}\\A
\end{minipage}
\begin{minipage}[b]{0.4\linewidth}
\centering
\includegraphics[width=0.5\linewidth]{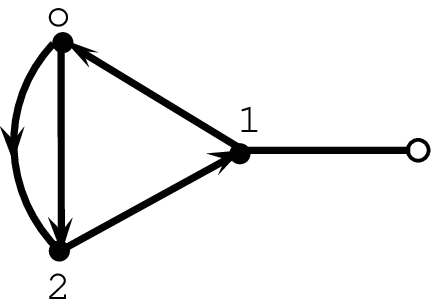}\\B
\end{minipage}
\caption{}\label{1ext}
\end{figure}
\end{itemize}
\end{itemize}
%note that Figure \ref{case12} comes from a spike and a triangle. We need to discuss if nodes \textbf{1,2} are connected to the same node $x\neq o$. Suppose \textbf{1,2} are both connected to a node $x$ distinct from $o$, then $\overline{1x}$ and $\overline{2x}$ must come from two spikes and the degree of $x$ must be two, see Figure \ref{notann}B. Otherwise, the graph is undecomposable. Suppose nodes \textbf{1,2} are not connected to the same nodes, we have Figure \ref{notann}C. Notice that in Figure \ref{notann}C, node $o$ and \textbf{1} can be connected to the same node. Same for node $o$ and \textbf{2}.\\\\
\begin{figure}[tb]
\centering
\begin{minipage}[b]{0.5\linewidth}
\centering
\includegraphics[width=0.5\linewidth]{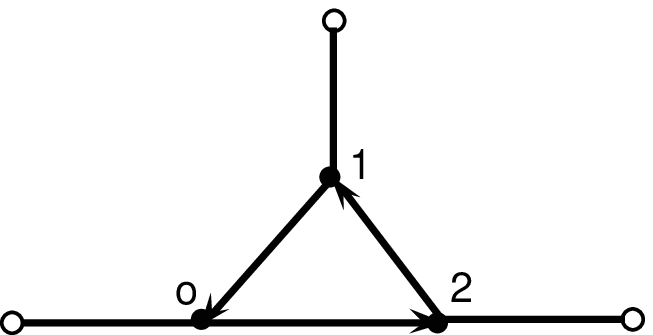}
\end{minipage}
%\vspace{0.5em}
%\mbox{\hspace{-4em}A\hspace{4em}B\hspace{4em}C}
\caption{}\label{n3}
\end{figure}
To sum up:\\
\textbf{1.} Every node in $G$ has degree at most 3.\\
\textbf{2.} Consider all nodes of degree 3. If all of them fall into
the decomposable categories, (Figure \ref{decomposable cases}) then either the neighborhood form disjoint connected component that can be easily decomposed, or we can apply corresponding replacement. If graph $G$ contains any neighborhood (up to a direction reversion on edges) that is unlisted in Figure \ref{decomposable cases}, the graph is not decomposable.\\
\begin{rmk}
We can reverse the directions of all edges to get another 14 neighborhoods in decomposable graph.
\end{rmk}
\begin{figure}[htbp]
\centering
  \begin{minipage}[b]{0.3\linewidth}
  \centering
  \hspace{-1.5em}\includegraphics[width=0.5\linewidth]{forkspike.eps}\\ \textbf{1}
  \end{minipage}
  \begin{minipage}[b]{0.3\linewidth}
  \centering
  \includegraphics[width=0.5\linewidth]{separateddiamond.eps}\\ \textbf{2}
  \end{minipage}
  \begin{minipage}[b]{0.3\linewidth}
  \centering
  \includegraphics[width=0.5\linewidth]{case111.eps}\\ \textbf{3}
  \end{minipage}\\
  \begin{minipage}[b]{0.3\linewidth}
  \centering
  \includegraphics[width=0.6\linewidth]{diamondtriangle.eps}\\ \textbf{4}
  \end{minipage}
  \begin{minipage}[b]{0.3\linewidth}
  \centering
  \includegraphics[width=0.6\linewidth]{trianglediamond.eps}\\ \textbf{5}
  \end{minipage}
  \begin{minipage}[b]{0.3\linewidth}
  \centering
  \includegraphics[width=0.7\linewidth]{case112deg3blackatp.eps}\\ \textbf{6}
  \end{minipage}\\
  \begin{minipage}[b]{0.3\linewidth}
  \centering
  \includegraphics[width=0.9\linewidth]{case112deg3atp.eps}\\ \textbf{7}
  \end{minipage}%
  \begin{minipage}[b]{0.3\linewidth}
  \centering
  \includegraphics[width=0.7\linewidth]{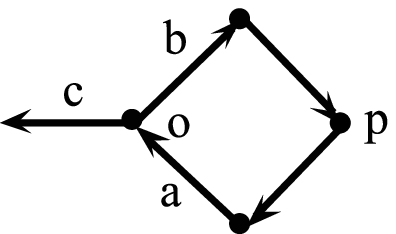}\\ \textbf{8}
  \end{minipage}
  \begin{minipage}[b]{0.3\linewidth}
  \centering
  \includegraphics[width=0.5\linewidth]{case12b2_1.eps}\\ \textbf{9}
  \end{minipage}\\
  \begin{minipage}[b]{0.3\linewidth}
  \centering
  \includegraphics[width=0.7\linewidth]{o2overlap.eps}\\ \textbf{10}
  \end{minipage}
  \begin{minipage}[b]{0.3\linewidth}
  \centering
  \includegraphics[width=0.8\linewidth]{case12b2.eps}\vspace{0.5em}\\ \textbf{11}
  \end{minipage}
  \begin{minipage}[b]{0.3\linewidth}
  \centering
  \includegraphics[width=0.8\linewidth]{case12b3_2.eps}\vspace{0.5em}\\ \textbf{12}
  \end{minipage}\\
  \begin{minipage}[b]{0.6\linewidth}
  \centering
  \includegraphics[width=0.3\linewidth]{case12b3_1.eps}
  \includegraphics[width=0.3\linewidth]{case12b3_1r.eps}
  \includegraphics[width=0.3\linewidth]{case12b3_1rr.eps}\\ \textbf{13}
  \end{minipage}
  \begin{minipage}[b]{0.3\linewidth}
  \centering
  \includegraphics[width=0.55\linewidth]{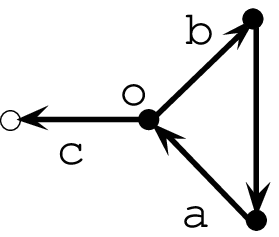}\\ \textbf{14}
  \end{minipage}\\
\caption{All decomposable cases for degree 3}\label{decomposable cases}
\end{figure}
\begin{rmk}Note that the degree of node
$o$ is not increased in any replacement.
\end{rmk}
In some of the above cases, the neighborhood of target node $o$ contains some other nodes of degree $3$. The algorithm covers the analysis of the neighborhood of these nodes in the following manner: \\
\begin{itemize}
\item For neighborhood 2 in Figure \ref{decomposable cases}, node $x$ has degree 3. The neighborhood of node $x$ is considered in the case derived from reversing the direction of edges in Figure \ref{decomposable cases}.2. Similarly, the neighborhood of node $p$ in neighborhood 4, 6 and that of node \textbf{2} in 9 and 10 is covered by reversing the directions of the edges in corresponding pictures.
\item The neighborhood of node $x$ in picture 5 of Figure \ref{decomposable cases} is covered by the one in picture 4. To be more specific, the neighborhood of $x$ in picture 5 is the neighborhood of $o$ in picture 4.
\item Picture 9 is a part of picture 10. Note the replacement for picture 9 is the same as the replacement in 10. Therefore, the order of replacement does not affect the result of the algorithm.
\item For node $1$ in picture 13, its neighborhood is the same as the one of node $o$ in picture 12. Since we don't apply any replacement for the neighborhood in picture 12, the order of examining nodes \textbf{1} and $o$ won't affect the result of algorithm. Similarly for node $1$ in picture 11.
\end{itemize}

\subsection{Identify the Decomposition}
In the previous section, we found all possible neighborhoods of nodes $o$ of degree 3 in a decomposable graph. In this section, we want to identify the neighborhood by checking two things:
\begin{itemize}
\item The number of nodes (other than nodes \textbf{1,2,3} and $o$) that are connected to
some of nodes \textbf{1,2,3}. Denote the number of such nodes by $n$. 
\item The direction of edges connecting $o$, its boundary nodes and other nodes that are connected to nodes \textbf{1,2,3}
\end{itemize}
If all three edges incident to $o$ have the same direction, the only possible neighborhood in a decomposable graph comes from
gluing a fork to a spike. Moreover, $n$ is $0$ or $1$. Suppose $n=1$, there is only node that differs from $o$ and is connected to nodes \textbf{1,2,3}. Denote it by $x$. If $x$ is connected to node \textbf{1} (resp. \textbf{2} or \textbf{3}), then edge $\overline{o1}$ (resp. $\overline{o2}$ or $\overline{o3}$)comes
from a spike and edges $\overline{o2},\overline{o3}$ (resp. $\overline{o1}$,$\overline{o3}$ or $\overline{o1},\overline{o2}$)come from a fork.\\\\
We focus on the case when there is one edge going towards node $o$ and two edges going away from node $o$. Note that the remaining case is when
there is one edge going away from node $o$ and two edges going towards node $o$. The latter case can be analyzed by reversing direction of all edges and using the following argument.\\\\
Assume node \textbf{1} is incident to the inward edge. Denote edges $\overline{o1},\overline{o2},\overline{o3}$ by $a,b,c$
respectively. By Figure \ref{decomposable cases}, $n\leq3$\\\\
Suppose $n=0$. By previous discussion, if the graph is decomposable, we can only have neighborhoods
as in Figure \ref{noextra}. After reversing all directions, we can get another four possible cases.\\\\
\begin{figure}[bth]
\centering
\begin{minipage}[c]{0.24\linewidth}
\centering
\includegraphics[width=0.7\linewidth]{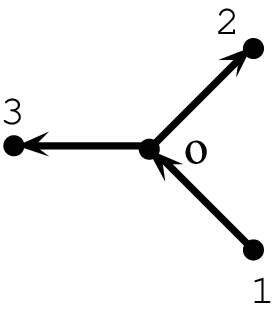}
\end{minipage}
\begin{minipage}[c]{0.24\linewidth}
\centering
\includegraphics[width=0.7\linewidth]{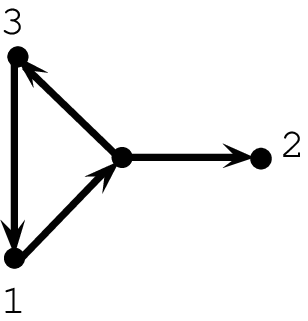}
\end{minipage}
\begin{minipage}[c]{0.24\linewidth}
\centering
\includegraphics[width=0.7\linewidth]{case12b2_1.eps}
\end{minipage}
\begin{minipage}[c]{0.24\linewidth}
\centering
\includegraphics[width=0.7\linewidth]{separateddiamond.eps}
\end{minipage}
\caption{}\label{noextra}
\end{figure}

\noindent Next, suppose $n=1$. Denote this node by $x$\\\\
Assume $x$ is connected to all nodes \textbf{1,2,3}. There are two cases:\\
 \begin{itemize}
 \item If deg(\textbf{1})=3, edges $b,c$ must come from the same diamond and edge $a$ comes from a triangle. See Figure \ref{x123} for directions of edges. Note that it's exactly picture 4 in Figure \ref{decomposable cases}.
 \item If deg(\textbf{1})=2, the neighborhood is a disjoint connected component, there are two possible decomposition: \textbf{1:} $a$ comes from a triangular block and $b,c$ come from a diamond; \textbf{2:} $a,b$ come from one triangular block, edges $\overline{2p}$ and $\overline{1p}$ come from another triangular block. edges $c$ and $\overline{3p}$ come from two spikes.
 \end{itemize}
\begin{figure}[tb]
\centering
\begin{minipage}[c]{0.5\linewidth}
\centering
\includegraphics[width=0.45\linewidth]{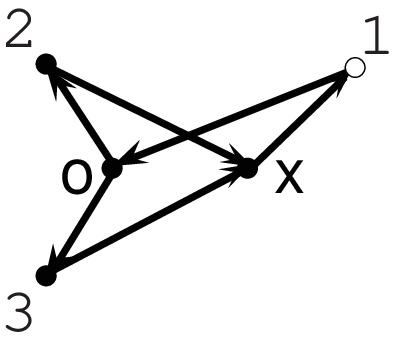}
\caption{}\label{x123}
\end{minipage}
\end{figure}
Suppose $x$ is connected to exactly two of nodes \textbf{1,2,3}. There are three cases:
\begin{enumerate}
\item Suppose $x$ is connected to nodes \textbf{1,2}. The neighborhood can only be as picture 4 or 5 in Figure \ref{decomposable cases}. Edges $a,b$ comes from the same triangle. Similarly, if $x$ is connected to nodes \textbf{1,3}, edges $a,c$ are in the same triangle.
\item Suppose $x$ is connected to nodes \textbf{2,3}. None of the pictures in Figure \ref{decomposable cases} contains such neighborhood. Hence the graph is undecomposable.
\end{enumerate}
\noindent If $x$ is connected to exactly one of the nodes \textbf{1,2,3}. First, suppose $x$ is connected to node \textbf{1} which is an endpoint of inward edge. Note that nodes \textbf{2} and \textbf{3} can not be connected to node \textbf{1}. Use argument in the previous section,  if the graph is decomposable, we conclude:
\begin{itemize}
\item If nodes \textbf{1,2} are connected, edges $a,b$ come from the same triangle and edge $c$ comes from a spike.
\item If nodes \textbf{1,3} are connected, edges $a,c$ come from the same triangle and edge $b$ comes from a spike.
\item If nodes \textbf{2,3} are both disconnected from node \textbf{1}, edges $b, c$ come from the same fork and edge $a$ comes from a spike.
\end{itemize}

\noindent Next suppose $x$ is connected to node \textbf{2}. If the graph is decomposable and node \textbf{1,2} are connected, then $a,b$ come from the same triangle. If nodes \textbf{1,2} are disconnected, then $a,c$ come from the same triangle. The criterion is similar if $x$ is connected only to node \textbf{3}.\\\\
Next, suppose $n=2$, denote these two corresponding nodes by $x,y$.\\\\
First, check if they are both connected to node \textbf{1}. If it's this case, neither node \textbf{2} or \textbf{3} can be connected to node \textbf{1}. Moreover, according to the argument in previous section, we have two cases. \textsf{1:} $a$ comes from a spike and $b,c$ comes from a fork; \textsf{2:} $a$ comes from a triangle. In the second case, $x,y$ must both be connected to nodes \textbf{2} or \textbf{3} by edges with compatible directions, and edges $a,b$ (edges $a,c$) are in the same triangle. See Figure \ref{decomposable cases} picture 5.\\\\
If $x,y$ are not both connected to node \textbf{1}, check if they are both connected to node \textbf{2}. If so, edges $a,c$ can only be obtained from a triangle and $b$ comes from a spike. Since $n=2$, there is no node other than $o$ that is connected to node \textbf{1} or \textbf{3}. Therefore, the neighborhood is as the one in Figure \ref{simple}. Note the neighborhood of $o$ is listed in picture 14 Figure \ref{decomposable cases}. The argument is similar if $x,y$ are both connected to node \textbf{3}.\\\\
\begin{figure}[btch]
\centering
\begin{minipage}[c]{0.5\linewidth}
\centering
\includegraphics[width=0.6\linewidth]{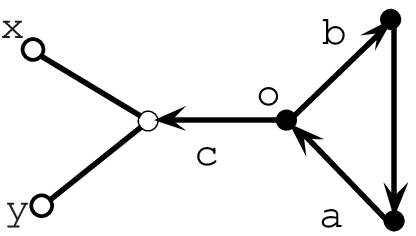}
\caption{}\label{simple}
\end{minipage}
\end{figure}
Next suppose $x,y$ are not connected to the same node. If the graph is decomposable, there are the following cases:
\begin{itemize}
\item $x$ is only connected to node \textbf{1} and $y$ only connected to node \textbf{2}. In this case, nodes \textbf{1,2} must be connected and $a,b$ come from the same triangle. It's similar if $x$ is only connected to node \textbf{1} and $y$ only connected to node \textbf{3}.
\item $x$ is only connected to nodes \textbf{1,2} and $y$ only connected to node \textbf{3}. In this case, $a,b$ come from the same triangle. If nodes \textbf{1,2} are connected, we have neighborhood shown in Figure \ref{decomposable cases} picture 13. If nodes \textbf{1,2} are disconnected, the neighborhood is shown in picture 8 (node $p=x$).
\item $x,y$ are connected to nodes \textbf{2,3} respectively. In addition, if nodes \textbf{1,2} are connected, then edges $a,b$ come from the same triangle. The neighborhood is as shown in picture 11. Similarly, if nodes \textbf{1,3} are connected, edges $a,c$ come from the same triangle. Notice that in this case, nodes \textbf{2} and \textbf{3} can not be connected to node \textbf{1} at the same time , neither can they both be disconnected at the same time.
\end{itemize}
Last of all, suppose $n=3$. Denote three corresponding nodes by $x,y,z$. According to the argument in previous section, in this case, the graph $G$ is decomposable only if node \textbf{1} is connected to node \textbf{2} or \textbf{3}, forming a triangle with the corresponding edges. See Figure \ref{n3}.\\\\

\begin{thm} Assume that every node in $G$ has degree less than or equal to
3. If all nodes of degree 3 fall into the cases listed in Figure
\ref{decomposable cases} (up to a reversion of edge directions), then $G$ is
decomposable. Otherwise, $G$ is undecomposable.
\end{thm}
\begin{proof} Assume that all degree 3 nodes are from Figure
\ref{decomposable cases} and all the necessary replacements have
been applied. Except for picture 2 and 6, which don't require replacement,
the replacements for all neighborhoods in Figure
\ref{decomposable cases} contain triangular blocks. Induction will
be used based on that.\\
Apply the corresponding replacement for all graph in
Figure \ref{decomposable cases}, and get $G'$. Notice that according
to the previous lemma, if $G'$ is decomposable, so is $G$. So
besides separated connected components: Graph 2,6, all node of
degree 3 in $G'$ are in the form of Figure \ref{case12}. Remove the
triangle as a block and use induction. After finitely many steps,
all nodes have degree at most 2. This can be obtained by gluing finite many
spikes.
\end{proof}
\vspace{1em} \noindent To conclude, if the graph $G$ is decomposable, we have exhausted
all possible neighborhoods of any node of degree at least 3. Any undecomposable neighborhood forces the whole graph to be
undecomposable. If none of these undecomposable neighborhoods is
contained in the graph, we apply necessary replacement to those of
degree 8,7,6,5 and 4 (in this exact order). These replacements
reduce the degrees of nodes and simplify the graph. In every step, it
is necessary to examine if any undecomposable neighborhood is
contained in the new graph. It is possible that after a step of
simplification, we obtain several connected components and
the same algorithm can be applied to each component. Eventually the
graph is reduced to the one with nodes of degree at most 3. The
possible neighborhoods of nodes of degree 3 are listed in Section
\ref{deg3node}. By the last theorem, we can determine if such graph is decomposable. And lemmas are
provided to show that all replacements
are reversible. Thus, in this case, the original graph is decomposable.\\
\begin{rmk} In most cases, the decomposition is unique. However, as mentioned in
the above argument, some neighborhood has non-unique decomposition.
As shown in Figure \ref{nonunique}, these neighborhoods are all
disjoint connected components. Each of them has finite many possible decompositions. We can reverse the direction of each picture to obtain another 14 neighborhoods with non-unique decompositions.
\begin{figure}[bt]
\centering
 \begin{minipage}[b]{0.3\linewidth}
 \centering
 \includegraphics[width=0.4\linewidth]{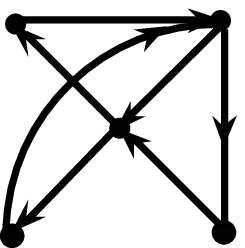}
 \end{minipage}
 \begin{minipage}[b]{0.3\linewidth}
 \centering
 \includegraphics[width=0.4\linewidth]{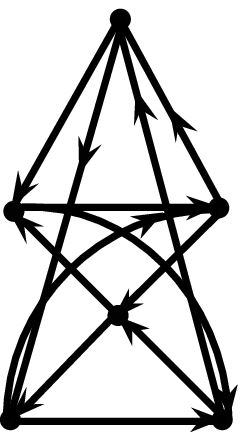}
 \end{minipage}
 \begin{minipage}[b]{0.3\linewidth}
 \centering
 \includegraphics[width=0.5\linewidth]{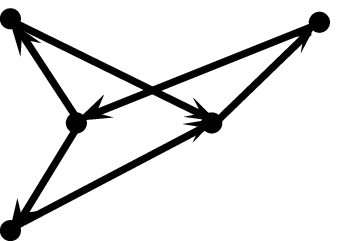}
 \end{minipage}\\
  \begin{minipage}[b]{0.3\linewidth}
 \centering
 \includegraphics[width=0.5\linewidth]{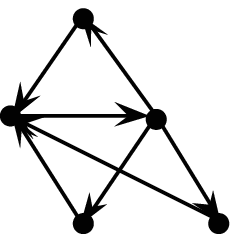}
 \end{minipage}
 \begin{minipage}[b]{0.3\linewidth}
 \centering
 \includegraphics[width=0.45\linewidth]{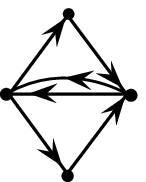}
 \end{minipage} \begin{minipage}[b]{0.3\linewidth}
 \centering
 \includegraphics[width=0.6\linewidth]{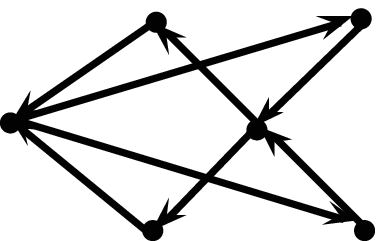}\\
 \vspace{1.8em}
 \end{minipage}\\
 \begin{minipage}[b]{0.3\linewidth}
 \centering
 \includegraphics[width=0.4\linewidth]{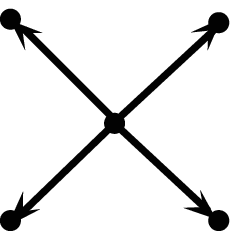}
 \end{minipage} \begin{minipage}[b]{0.3\linewidth}
 \centering
 \includegraphics[width=0.4\linewidth]{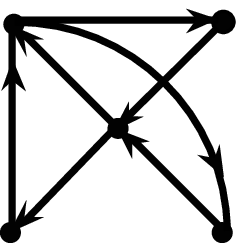}
 \end{minipage} \begin{minipage}[b]{0.3\linewidth}
 \centering
 \includegraphics[width=0.4\linewidth]{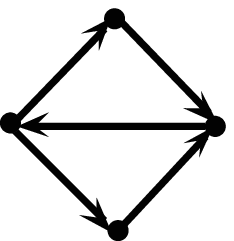}
 \end{minipage}\\
 \begin{minipage}[b]{0.3\linewidth}
 \centering
 \includegraphics[width=0.4\linewidth]{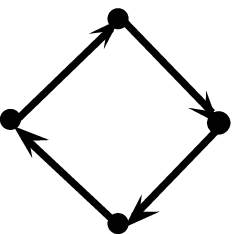}
 \end{minipage}
 \begin{minipage}[b]{0.3\linewidth}
 \centering
 \includegraphics[width=0.35\linewidth]{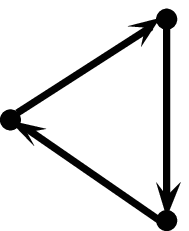}
 \end{minipage}
 \begin{minipage}[b]{0.3\linewidth}
 \centering
 \includegraphics[width=0.27\linewidth]{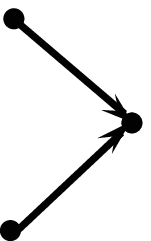}
 \end{minipage}\\
 \begin{minipage}[b]{0.3\linewidth}
 \centering
 \includegraphics[width=0.5\linewidth]{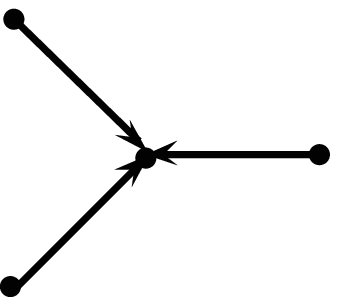}
 \end{minipage}
 \begin{minipage}[b]{0.3\linewidth}
 \centering
 \includegraphics[width=0.5\linewidth]{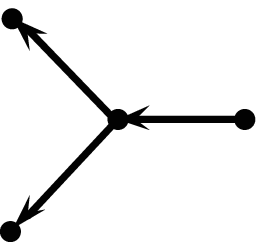}
 \end{minipage}\\
\caption{Neighborhoods that Have Non-unique Decomposition}
\label{nonunique}
\end{figure}
\end{rmk}
\begin{rmk} At each stage of simplification, we apply replacements for at most as many as the number of nodes in the graph.
Moreover, if the neighborhood of a node needs replacement in the algorithm, after applying the replacement, the degree of the considered node
becomes 3. According to the algorithm, this means we apply replacement for at most once to the same nodes, that is, to reduce the degree to 3 if it's not 3 in the original graph. The number of replacement less than the number of nodes in the graph.
This is noticed by P.Tumarkin. In additoin, this algorithm provides a fast way to determine when a quiver of size
larger than 10 has finite mutation type. (See \cite{FST2} for
detail.)
\end{rmk}
\section*{Acknowledgement}
\indent I thank P.Tumarkin for the discovery of the linearity of
this algorithm, S.Fomin, M.Shapiro and D.Thurston for bringing up
this problem in \cite{FST1}. I especially thank my advisor
Dr.Shapiro for helpful advises and inspiring discussions, and for
kindly providing proofreading of this paper.

\end{document}